\documentclass[hidelinks,onefignum,onetabnum]{siamart220329}
\usepackage{bm}
\usepackage{amsfonts}
\usepackage{amssymb}
\usepackage{graphicx}
\usepackage{epstopdf}
\usepackage{algorithmic}
\usepackage{enumitem}
\usepackage{booktabs}
\usepackage{multirow}
\usepackage{amsopn}
\usepackage{changes}
\ifpdf
\DeclareGraphicsExtensions{.eps,.pdf,.png,.jpg}
\else
\DeclareGraphicsExtensions{.eps}
\fi

\newsiamremark{remark}{Remark}
\newsiamthm{exper}{Experiment}
\newsiamremark{hypothesis}{Hypothesis}
\crefname{hypothesis}{Hypothesis}{Hypotheses}
\newsiamthm{claim}{Claim}

\headers{New JD type SVD methods}{Jinzhi Huang and Zhongxiao Jia}

\allowdisplaybreaks[4]

\title{
	New Jacobi--Davidson type
	methods for the large SVD computations\thanks{Submitted to the editors DATE.
	\funding{The work of the first author was supported by the Youth
		Fund of the National Science Foundation of China under grant
		12301485, the Youth Program of the Natural Science Foundation
		of Jiangsu Province under grant BK20220482 and the Jiangsu Province
		Youth Science and Technology Talent Support Program under grant JSTJ-2025-828,
		and the work of the second author was supported by the National Science	Foundation
		of China under grant 12571404}.}}

\author{Jinzhi Huang\thanks{School of Mathematical Sciences,
		Soochow University, 215006 Suzhou, China
		(\url{jzhuang21@suda.edu.cn}).}
\and
	Zhongxiao Jia\thanks{Corresponding author. Department of
		Mathematical Sciences, Tsinghua University, 100084 Beijing,
		China (\url{jiazx@tsinghua.edu.cn}).}}

\DeclareMathOperator{\new}{new}
\DeclareMathOperator{\out}{out}
\DeclareMathOperator{\diag}{diag}

\DeclareMathOperator{\spans}{span}

\newcommand{\UU}{\mathcal{U}}
\newcommand{\VV}{\mathcal{V}}
\newcommand{\RR}{\mathcal{R}}

\newcommand{\bsmallmatrix}[1]{\begin{bmatrix}\begin{smallmatrix}
			#1\end{smallmatrix}\end{bmatrix}}

\begin{document}
	\maketitle
	
	\begin{abstract}
		In a Jacobi--Davidson (JD) type method for singular value decomposition
		(SVD) problems, called JDSVD, a large symmetric and generally indefinite
		correction equation is solved iteratively at each outer iteration, which
		constitutes the inner iterations and dominates the overall efficiency of JDSVD.
		In this paper, by fully exploiting useful information
from current subspaces,
a new effective correction equation is derived at each outer iteration, leading to
a new variant of JDSVD, called JDSVD-V.
It is proved that JDSVD-V retains the same convergence of the outer iterations as JDSVD.
A substantial advantage of JDSVD-V over JDSVD is that the new
correction equations in JDSVD-V are much easier to iteratively
 solve than the standard ones in JDSVD: the MINRES method for the new correction equations
		converges much faster when there
		is a cluster of singular values closest to a given target, a typical case
in applications.
		A new thick-restart JDSVD-V algorithm with deflation and purgation is
		proposed that simultaneously accelerates the outer and inner convergence
		of the standard thick-restart JDSVD and computes several singular triplets.
		Numerical experiments justify the theory and illustrate the considerable
		superiority of JDSVD-V to JDSVD, and demonstrate that a similar two-stage
		JDSVD-V algorithm substantially outperforms the most advanced PRIMME\_SVDS
		software nowadays for computing the smallest singular triplets.
	\end{abstract}
	
	
	
	\section{Introduction}\label{sec:1}
	Let $A\in\mathbb{R}^{M\times N}$ be large and possibly sparse
	with $M\geq N$, and its singular value decomposition (SVD) be
	\begin{equation}\label{svdA}
		A = U\Sigma V^T,
	\end{equation}	
	where the superscript $T$ denotes the transpose of a matrix,
	$U=[u_1,\dots,u_M]\in\mathbb{R}^{M\times M}$ and
	$V=[v_1,\dots,v_N]\in\mathbb{R}^{N\times N}$ are orthogonal,
	and $\Sigma=\diag\{\sigma_1,\dots,\sigma_N\}\in\mathbb{R}^{M\times N}$
	with $\sigma_i\geq0,i=1,\dots,N$.
	The $(\sigma_i,u_i,v_i)$,  $i=1,\dots,N$ are the singular triplets
	of $A$ with the singular values $\sigma_i$ and
the associated left and right singular vectors $u_i,v_i$.
	
	Given a target $\tau\geq 0$ that is not equal to any singular value
	of $A$, assume that the singular values of $A$ are labeled by
	\begin{equation}\label{svalues}
		|\sigma_1-\tau|<\dots<|\sigma_{\ell}-\tau|<
		|\sigma_{\ell+1}-\tau|\leq\dots\leq|\sigma_N-\tau|.
	\end{equation}
	We are interested in the $\ell$ singular triplets $(\sigma_i,u_i,v_i)$,
	$i=1\dots,\ell$.
	Specifically, if $\tau$ is inside $(\sigma_{\min},\sigma_{\max})$
	but neither close to $\sigma_{\min}$ nor
	to $\sigma_{\max}$ with $\sigma_{\min}$ and $\sigma_{\max}$ being the
smallest and largest singular values of $A$, then $(\sigma_i,u_i,v_i)$, $i=1\dots,\ell$
	are called interior singular triplets of $A$; otherwise they
	are called extreme, i.e., the largest or smallest, ones.
	Particularly, $\tau=0$ corresponds to the smallest singular triplets and
	$\tau$ closest to $\sigma_{\max}$ corresponds to the largest ones.
	
	For the concerning large SVD problem, only
	projection methods can be used. Over the years,
	projection methods \cite{golub2012matrix,parlett1998symmetric,
		saad2003,stewart2001matrix,vandervorst} for large eigenvalue problems
	have been adapted to the computation of a partial SVD of $A$. Given left and
	right searching subspaces, there are four classes of projection methods
	for computing approximate singular triplets of $A$, which are based on
	the standard extraction \cite{berry1992large,hernandes2008robust,hochstenbach2001jacobi,
		hochstenbach2004harmonic,jia2003implicitly,wu2015preconditioned},
	the harmonic extraction \cite{hochstenbach2001jacobi,
		hochstenbach2004harmonic,hochstenbach2008harmonic,
		jia2005convergence,kokiopoulou2004computing,niu2012harmonic},
	the refined extraction \cite{hochstenbach2004harmonic,
		hochstenbach2008harmonic,huang2019inner,jia2003implicitly,
		wu16primme,wu2015preconditioned}, and the refined harmonic extraction
	\cite{huang2019inner,jia2010refined,wu16primme,wu2015preconditioned}, respectively.
	For these methods, a common core task is the construction of left
	and right searching subspaces so that they contain
	as much information on the desired singular vectors as possible.
	Besides a very recent approach that is used in the FEAST SVDsolvers
	\cite{jiazhang2023a,jiazhang2023b}
	for computing the singular triplets with all the singular values in a given interval,
	there are mainly two types of approaches for this task.
	The first approach is the Lanczos bidiagonalization (LBD) process
	\cite{berry1992large,golub2012matrix}. It reduces $A$ to a sequence
	of small bidiagonal matrices and generates two sets of orthonormal
	bases of the left and right Krylov subspaces formed by $AA^T$ and
	$A^TA$ with relevant starting vectors. Such searching
	subspaces generally favor the left and right singular
	vectors corresponding to the extreme singular values, and the distances between these singular vectors and the
	corresponding Krylov subspaces generally tend to zero first as the LBD process proceeds.
	Combined with appropriate extraction approaches,
	state-of-the-art LBD type algorithms \cite{jia2003implicitly,jia2010refined,kokiopoulou2004computing,
		larsen2001combining,sorensen1992implicit} can compute
	the extreme singular triplets of $A$ efficiently if the desired singular values
	are not many and are well separated.
	Unfortunately, the smallest singular values from applications are often clustered,
	so that the LBD type methods may work poorly.
	
The Jacobi--Davidson (JD)-type methods, proposed by Hochstenbach \cite{hochstenbach2001jacobi}
for large SVD computations and referred to as the JDSVD method, give
	the second approach to constructing subspaces.
Wu and Stathopoulos \cite{wu2015preconditioned} present a
hybrid two-stage method, called PHSVD.
At the first stage, PHSVD transforms the SVD problem of $A$
into the mathematically equivalent eigenproblem of $A^TA$ and computes the desired singular
values and right singular vectors to best achievable accuracy
by a chosen projection method. Then if further accuracy is required,
it switches to the second stage that solves the eigenproblem of
$\bsmallmatrix{\bm{0}&A\\A^T&\bm{0}}$ by a chosen projection method
using the approximate singular vectors obtained in the first stage
as initial guesses.
Wu, Romero and Stathopoulos \cite{wu16primme} develop the PRIMME$\_$SVDS
software programmed in C language;
it implements the PHSVD method with the refined extraction in the second stage and
uses JDQMR by default,
while the method in the first stage is taken as the default
`DYNAMIC' that dynamically switches between JDQMR and the GD+k
method to minimize the time, where QMR in JDQMR means that a symmetric
QMR (sQMR) method is used to solve the correction equations \cite{stathsisc2007}.
A great merit of the hybrid PRIMME$\_$SVDS is that, regarding
the outer iterations, JDQMR or GD+k in the first stage is generally much faster than it is
in the second stage,
which is true when the smallest singular values are desired.
Goldenberg, Stathopoulos and Romero  \cite{goldenberg2019} present a Golub--Kahan Davidson method,
which is a specific JDSVD method and is competitive with PRIMME$\_$SVDS
when computing the smallest singular triplets.

At each step of JDSVD \cite{hochstenbach2001jacobi}
and the second stage of PRIMME$\_$SVDS, one must
solve a certain correction equation that involves
$\bsmallmatrix{-\tau I&A\\A^T&-\tau I}$ and the current approximate singular triplet,
and uses its (approximate) solution to expand the searching subspaces. Whenever the current approximate
singular value starts to converge with fair accuracy,
the two methods will replace $\tau$ in the correction equations so
that their solutions provide richer information on the desired left and right singular vectors.
Such kind of correction equations is approximately solved iteratively,
called the inner iterations, while the outer iterations
compute approximate singular triplets with respect to the
searching subspaces. The least solution accuracy requirement on the
correction equations and convergence speed of the inner iterations
critically affect the outer convergence and dominate the
overall efficiency of JDSVD and PRIMME$\_$SVDS.
The least solution accuracy means that the outer iterations of
the resulting (inexact) JDSVD and PRIMME$\_$SVDS methods behave
as if all the correction equations were solved exactly.
By extending the results in \cite{jia2014inner} on a JD-type method
for the eigenvalue problem to the JDSVD method,
the authors in \cite{huang2019inner} have figured out such accuracy requirement
and proved that it suffices and is safe to solve all the
correction equations with the relative errors $\varepsilon\in [10^{-4}, 10^{-3}]$
of approximate solutions. Precisely, suppose that the JDSVD method
converges very fast and only around ten outer iterations are needed
when the corrections equations are solved exactly, then
it is proved in \cite{huang2019inner,jia2014inner} that
the approximate solutions of the correction equations
with the relative errors around $10^{-4}$ suffice to ensure that the resulting JDSVD method uses
almost the same outer
iterations to converge; if the JDSVD method
converges slowly and uses a lot of outer iterations, then the approximate solutions of
the correction equations with the relative errors
$\varepsilon\in [10^{-3},10^{-2}]$ can make the resulting
inexact JDSVD method well mimic the outer iterations of the exact JDSVD method.

For symmetric positive (or negative) definite and indefinite linear systems,
the CG and MINRES methods are preferences among Krylov iterative solvers, respectively.
Their convergence rates are uniquely
determined by the eigenvalues of an underlying coefficient matrix  \cite{greenbaum1997iterative}.
More precisely, the MINRES method may
converge very slowly when the coefficient matrix has relatively
small eigenvalues in magnitude; if those small
eigenvalues disappear, MINRES will converge (much) faster.
As it will turn out, for $\sigma_1$ close to $\tau$,
if the orthogonal projector in the standard correction
equation of JDSVD is replaced by a new one that is properly formed by approximations to the left and right
singular vectors $u_i$ and $v_i$ associated with the $m(>1)$ clustered singular values
$\sigma_i,i=1,2,\ldots,m$ rather than only $\sigma_1$,
then the corresponding coefficient matrix in the new correction equation
has no small eigenvalues in magnitude, thereby making MINRES converge much faster than it does
for the standard correction equation.

Strikingly, extensive numerical experiments have demonstrated that, as
the outer iterations of JDSVD proceed, the
subspaces simultaneously provide increasingly better approximations to
the left and right singular vectors $u_i$ and $v_i$, $i=1,\ldots,m$ corresponding to the $m$
clustered singular values $\sigma_i$ of $A$. We will fully exploit such basic fact to derive
new  correction equations and obtain a new variant of JDSVD,
called JDSVD-V. We particularly mention that any
preconditioner for the standard correction equations in
JDSVD are directly applicable to the new ones in JDSVD-V.

In this paper, we first focus on the computation of one singular
triplet $(\sigma_1,u_1,v_1)$, supposing that there are $m$ clustered singular
values of $A$ closest to the target $\tau$.
For the new correction equation, we prove an extremely important property:
its solution has the same quality as that of the standard correction equation
when expanding the current subspaces, meaning that the solution of the new
correction equation expands the subspaces as effectively as that of the
standard one and, therefore, does not slow down the convergence of the subsequent outer iterations.
Then we establish convergence results on MINRES for the new and standard correction
equations, proving that MINRES converges substantially faster for the former
than it does for the latter once $\sigma_1$ is clustered with $\sigma_2,\ldots,\sigma_m$
for $m>1$, a typical case in applications, especially
when the smallest singular triplets are desired. As a byproduct,
we show why the correction equation proposed in \cite{genseberger1999alternative} for the eigenvalue
problem and adapted to the current context
is less effective than those in JDSVD and JDSVD-V and the corresponding
JDSVD variant uses more outer iterations to converge.
The reasoning applies to the corresponding JD method for the eigenvalue problem.

In computations, we shall set up quantitative criteria and
propose an adaptive approach to dynamically form the new correction equation at each outer iteration.
For practical purpose, we propose a new thick-restart JDSVD-V algorithm
with certain deflation and purgation techniques to compute the $\ell(>1)$ singular triplets of $A$.

Since the goal of this paper is to set up new correction equations, we shall take
the standard extraction-based JDSVD
method as an instance and the shifts in all the correction equations
as the fixed $\tau$. However, the theoretical results and analysis on
the new and standard correction equations are adaptable
to the harmonic, refined and refined harmonic
extractions-based JDSVD methods and their adaptively-changing-shift
variants \cite{huang2019inner}, so are the corresponding
thick-restart algorithms. They are also adaptable to
the eigenproblem-based JDSVD methods used in
PRIMME$\_$SVDS \cite{wu16primme}.

In  \cref{sec:2} we review the standard JDSVD method.
In \cref{sec:4} we derive the new correction
equation in detail, make a convergence analysis on
MINRES for it, and consider
how to dynamically set up new correction equations in
computations. In  \cref{sec:5} we propose the thick-restart
JDSVD-V with deflation and purgation to compute the $\ell$ singular triplets of $A$.
Numerical experiments in  \cref{sec:6} illustrate the great superiority
of the new thick-restart JDSVD-V to the standard thick-restart JDSVD.
To compare our algorithm with the nowadays
most advanced PRIMME$\_$SVDS software \cite{wu16primme}, in the same spirit we have developed
a hybrid two-stage thick-restart JDSVD-V algorithm,
called JDSVD-V\_HYBRID: in the first stage, a similar JDSVD-V method
computes the desired largest or smallest singular triplets based on
the eigenproblem of $A^TA$ (or $AA^T$ if $m<n$) with the same outer
stopping tolerance as that in the first stage of PRIMME\_SVDS,
and it then switches to the second stage  whenever necessary, where JDSVD-V
is based on the eigenproblem of $\bsmallmatrix{\mathbf{0}&A\\A^T&\mathbf{0}}$
using the approximate singular triplets computed at the first stage as
initial guesses. We demonstrate
that JDSVD-V\_HYBRID substantially outperforms PRIMME\_SVDS
when computing the smallest singular triplets and is at least competitive with
the latter when computing the largest ones.
Section~\ref{sec:7} concludes the paper.

Throughout this paper, denote by $\|X\|$ the 2-norm of a vector or matrix $X$,
by $\kappa(X)=\sigma_{\max}(X)/\sigma_{\min}(X)$ the condition number of $X$ with $\sigma_{\max}(X)$ and
$\sigma_{\min}(X)$ the largest and smallest singular values of $X$, respectively.

\section{The basic JDSVD method}\label{sec:2}
We review the basic standard extraction-based JDSVD method \cite{hochstenbach2001jacobi} for
computing the singular triplet $(\sigma_1,u_1,v_1)$ of $A$.

At the $k$th outer iteration, given two
$k$-dimensional left and right searching subspaces $\UU\subset\mathbb{R}^{M}$
and $\VV\subset\mathbb{R}^{N}$, the standard extraction approach seeks
for scalars $\theta_i\geq 0$ and unit-length vectors $\tilde u_i\in\UU$
and $\tilde v_i\in\VV$ that satisfy the double-orthogonality
\begin{equation}\label{standard}
	\qquad A\tilde v_i-\theta_i \tilde u_i\perp \UU
	\quad\mbox{and}\quad
	A^T\tilde u_i-\theta_i \tilde v_i\perp \VV, \qquad i=1,\dots,k.
\end{equation}
The $(\theta_i,\tilde u_i,\tilde v_i)$ are the Ritz approximations
of $A$ from $\UU$ and $\VV$ with $\theta_i$ the Ritz values
and $\tilde u_i$ and $\tilde v_i$ the left and right Ritz vectors, respectively.
Condition \eqref{standard} is denoted as
$
(A\tilde v_i-\theta_i \tilde u_i,A^T\tilde u_i-\theta_i \tilde v_i)
\perp\perp (\UU,\VV),\ i=1,\dots,k.
$
Label $|\theta_1-\tau|\leq |\theta_2-\tau|\leq\dots\leq|\theta_k-\tau|$.
The method takes $(\theta_1,\tilde u_1,\tilde v_1)$ as an approximation
to $(\sigma_1,u_1,v_1)$.

Let $\widetilde U\in\mathbb{R}^{M\times k}$ and $\widetilde
V\in\mathbb{R}^{N\times k}$ be orthonormal basis
matrices of $\UU$ and $\VV$. Then the left and right Ritz vectors
$\tilde u_i=\widetilde Uc_i$ and $\tilde v_i=\widetilde Vd_i$ with
unit-length $c_i\in\mathbb{R}^{k}$ and $d_i\in\mathbb{R}^{k}$, $i=1,\dots,k$.
Define $H=\widetilde U^TA\widetilde V$. Then
\eqref{standard} is equivalent to
\begin{equation*}
	\qquad Hd_i=\theta_i c_i
	\qquad\mbox{and}\qquad
	H^Tc_i=\theta_i d_i,
	\qquad i=1,\dots,k.
\end{equation*}
This means that the $(\theta_i,c_i,d_i),i=1,\dots,k$  are the singular triplets of $H$
and
\begin{equation}\label{appsingular}
\qquad	(\theta_i,\tilde u_i,\tilde v_i)
=(\theta_i,\widetilde Uc_i,\widetilde Vd_i), \qquad i=1,\dots,k,
\end{equation}
whose residuals are defined by
\begin{equation}\label{residual}
	r_i = r(\theta_i,\tilde u_i,\tilde v_i) =
	\begin{bmatrix}A\tilde v_i-\theta_i \tilde u_i\\ A^T\tilde u_i-\theta_i \tilde v_i\end{bmatrix}.
\end{equation}
By \eqref{standard}, we particularly have
\begin{equation}\label{resperp}
	\qquad r_i\perp\perp (\tilde u_j,\tilde v_j), \qquad i,j=1,\dots,k.
\end{equation}
Obviously, $r_i=\bm{0}$ if and only if $(\theta_i,\tilde u_i,\tilde v_i)$
is an exact singular triplet of $A$.
For a user prescribed tolerance $\varepsilon_{\mathrm{out}}>0$, we claim that
$(\theta_1,\tilde u_1,\tilde v_1)$
has converged, and terminate the basic JDSVD method if its residual $r_1$ satisfy
\begin{equation}\label{convergence}
	\|r_1\|\leq \|A\| \cdot \varepsilon_{\mathrm{out}}.
\end{equation}

If $(\theta_1,\tilde u_1,\tilde v_1)$ does not yet converge, JDSVD expands
$\UU$ and $\VV$ in the following way: It
first solves the large symmetric and generally indefinite
correction equation
\begin{equation}\label{correction}
	\begin{bmatrix}	I-\tilde u_1\tilde u_1^T&\\&I-\tilde v_1\tilde v_1^T \end{bmatrix}
	\begin{bmatrix} -\tau I&A\\A^T&-\tau I \end{bmatrix}
	\begin{bmatrix}	I-\tilde u_1\tilde u_1^T&\\&I-\tilde v_1\tilde v_1^T \end{bmatrix}
	\begin{bmatrix} s\\ t\end{bmatrix}
	=-r_1
\end{equation}
for $(s,t)\perp\perp(\tilde u_1,\tilde v_1)$ approximately using, e.g.,
the preferable MINRES method, and then orthonormalizes the approximate
solutions, denoted by
$\tilde s$ and $\tilde t$, against $\widetilde U$ and $\widetilde V$
to obtain $u_+$ and $v_+$.
The orthonormal columns of
\begin{equation*}
	\widetilde U_{\new} = [\widetilde U, u_{+}]
	\qquad\mbox{and}\qquad
	\widetilde V_{\new} = [\widetilde V, v_{+}]
\end{equation*}
construct the expanded $(k+1)$-dimensional left and right subspaces
\begin{equation*}
	\UU_{\new}=\spans\{\widetilde U_{\new}\}
	=\spans\{\widetilde U,\tilde s\}
	\quad\mbox{and}\quad
	\VV_{\new}=\spans\{\widetilde V_{\new}\}
	=\spans\{\widetilde V,\tilde t\}.
\end{equation*}
JDSVD then computes a new approximation to
$(\sigma_1,u_1,v_1)$ from $\UU_{\new}$ and $\VV_{\new}$.

The shift-and-invert residual Arnoldi (SIRA) method \cite{jia2014inner} applied to
the current context solves the inner linear system
\begin{equation}\label{correction5}
	\begin{bmatrix} -\tau I&A\\A^T&-\tau I \end{bmatrix}
	\begin{bmatrix} s\\ t\end{bmatrix}
	=-r_1,
\end{equation}
and $s$ and $t$ are orthonormalized against to $\widetilde U$ and $\widetilde V$
to obtain the expansion vectors $u_+$
and $v_+$. It is proved in \cite[Theorem 2.1]{jia2014inner}
that such $u_+$ and $v_+$ are identical to the ones in JDSVD when $\tau\not=\theta_1$
and SIRA is thus mathematically equivalent to JDSVD. Notice that the solution
of \eqref{correction5} performs one-step inverse iteration on
$\begin{bmatrix} -\tau I&A\\A^T&-\tau I \end{bmatrix}$.
Computationally, however, it is
faster to iteratively solve \eqref{correction} than \eqref{correction5} once $\theta_1$ starts to converge,
indicating that JDSVD is preferable to SIRA \cite{jia2014inner}.

Another alternative correction equation in \cite{deSturler2002ImprovingTC,genseberger1999alternative}
applied to our current SVD context is
\begin{equation}\label{correctionV}
	\begin{bmatrix}	I-\tilde U\tilde U^T&\\&I-\tilde V\tilde U^T \end{bmatrix}
	\begin{bmatrix} -\tau I&A\\A^T&-\tau I \end{bmatrix}
	\begin{bmatrix}	I-\tilde U\tilde U^T&\\&I-\tilde V\tilde V^T \end{bmatrix}
	\begin{bmatrix} s\\ t\end{bmatrix}
	=-r_1
\end{equation}
for $(s,t)\perp\perp(\tilde U,\tilde V)$. We will come back to this equation in the next
section and show why it is less effective than \eqref{correction}.

The convergence results to be presented shall indicate
that MINRES may converge very slowly for \eqref{correction}
when $\sigma_1$ is clustered with some
other singular values of $A$, which is proved to correspond to the case
that the coefficient matrix, denoted by $\widetilde{B}_1$ from now on, in
\eqref{correction}
has a cluster of small eigenvalues in magnitude.
Strikingly, as JDSVD proceeds, the subspaces $\UU$ and $\VV$ contain
increasingly more accurate approximations to the singular vectors of $A$
with  {\em all those} singular values closest to $\tau$, and are shown to
provide rich information on the eigenvectors of $\widetilde B_1$
associated with all the smallest eigenvalues in magnitude.  As a result,
these facts motivate us to fully exploit such information
to set up a new correction equation that is easier to be solved by MINRES
and, meanwhile, its solution expands $\UU$ and $\VV$ as effectively as that of \eqref{correction}.

\section{The JDSVD-V method}\label{sec:4}

In \eqref{svalues}, suppose that there are $m$ clustered singular
values of $A$ closest to $\tau$:
\begin{equation}\label{svalues22}
	|\sigma_1-\tau|\leq\dots \leq |\sigma_{m}-\tau|\ll
	|\sigma_{m+1}-\tau|\leq\dots\leq|\sigma_N-\tau|.
\end{equation}
From \eqref{svdA}, we denote
\begin{equation}\label{blocksvd}
	\Sigma_m=\diag\{\sigma_1,\dots,\sigma_m\},\qquad
	U_m=[u_{1},\dots,u_{m}],\qquad
	V_m=[v_{1},\dots,v_{m}].
\end{equation}
$(\Sigma_m,U_m,V_m)$ is a partial
SVD of $A$ with the $m$ clustered singular values closest to $\tau$.

For the Ritz approximations in \eqref{appsingular} computed at the $k$th outer iteration of JDSVD with $k>m$,
where $(\theta_1,\tilde u_1,\tilde v_1)$ is an approximation to the desired $(\sigma_1,u_1,v_1)$,
take $(\theta_i,\tilde u_i,\tilde v_i)$, $i=2,\dots,m$ to be approximations to the singular triplets
$(\sigma_i,u_i,v_i)$, $i=2,\dots,m$ of $A$.
Denote
\begin{equation}\label{UUVVtilde}
	\Theta_m=\diag\{\theta_1,\dots,\theta_m\},\qquad
	\widetilde U_m=[\tilde u_1,\dots,\tilde u_m],\qquad
	\widetilde V_m=[\tilde v_1,\dots,\tilde v_m].
\end{equation}
Then $(\Theta_m,\widetilde U_m,\widetilde V_m)$ is an
approximation to $(\Sigma_m,U_m,V_m)$.

\subsection{A new correction equation}\label{subsec1}
By \eqref{appsingular}--\eqref{resperp} and \eqref{UUVVtilde}, we have
\begin{equation*}
	r_1\perp\perp\left(\widetilde U_m,\widetilde V_m \right) \mbox{,\ \ i.e.,\ \ }
	r_1=\diag\{I-\widetilde U_m\widetilde U_m^T,I-\widetilde V_m\widetilde V_m^T\}r_1.
\end{equation*}
Therefore, premultiplying \eqref{correction} by
$\diag\{I-\widetilde U_m\widetilde U_m^T,I-\widetilde V_m\widetilde V_m^T\}$ yields
\begin{equation}\label{correction3}
	\begin{bmatrix}	I-\widetilde U_m\widetilde U_m^T&\\
		&I-\widetilde V_m\widetilde V_m^T \end{bmatrix}
	\begin{bmatrix} -\tau I&A\\A^T&-\tau I \end{bmatrix}
	\begin{bmatrix} s\\ t\end{bmatrix}
	=-r_1,
\end{equation}
Writing
\begin{equation*}
	\begin{bmatrix} s\\ t\end{bmatrix}
	= \begin{bmatrix}I-\widetilde U_m\widetilde U_m^T&\\
		&I-\widetilde V_m\widetilde V_m^T \end{bmatrix}
	\begin{bmatrix} s\\ t\end{bmatrix}+	
	\begin{bmatrix}	 \widetilde U_m\widetilde U_m^T&\\
		& \widetilde V_m\widetilde V_m^T \end{bmatrix}
	\begin{bmatrix} s\\ t\end{bmatrix}
\end{equation*}
and substituting it into \eqref{correction3}, we obtain
\begin{equation}\label{correction4}
	\begin{bmatrix}	I-\widetilde U_m\widetilde U_m^T\!\!\!\!&\\
		&\!\!\!\!I-\widetilde V_m\widetilde V_m^T \end{bmatrix}
	\begin{bmatrix} -\tau I&A\\A^T&-\tau I \end{bmatrix}
	\begin{bmatrix}	I-\widetilde U_m\widetilde U_m^T\!\!\!\!&\\
		&\!\!\!\!I-\widetilde V_m\widetilde V_m^T \end{bmatrix}
	\begin{bmatrix} s\\ t\end{bmatrix}
	=-r_1-r_{\mathrm{tail}},
\end{equation}
where, by $\Theta_m=\widetilde U_m^TA\widetilde V_m$, the tail term
\begin{eqnarray}
	r_{\mathrm{tail}}
	&=&\begin{bmatrix}	I-\widetilde U_m\widetilde U_m^T&\\
		&I-\widetilde V_m\widetilde V_m^T \end{bmatrix}
	\begin{bmatrix} -\tau I&A\\A^T&-\tau I \end{bmatrix}
	\begin{bmatrix}	 \widetilde U_m\widetilde U_m^T&\\
		& \widetilde V_m\widetilde V_m^T \end{bmatrix}
	\begin{bmatrix} s\\ t\end{bmatrix} \nonumber \\
	&=&\begin{bmatrix}	I-\widetilde U_m\widetilde U_m^T&\\
		&I-\widetilde V_m\widetilde V_m^T \end{bmatrix}
	\begin{bmatrix} -\tau I&A\\A^T&-\tau I \end{bmatrix}
	\begin{bmatrix}	 \widetilde U_m&\\& \widetilde V_m \end{bmatrix}
	\begin{bmatrix} \widetilde U_m^Ts\\ \widetilde V_m^Tt\end{bmatrix} \nonumber\\
	&=&\begin{bmatrix} & A\widetilde V_m-\widetilde U_m\Theta_m \\ A^T\widetilde
		U_m-\widetilde V_m\Theta_m\end{bmatrix}
	\begin{bmatrix} \widetilde U_m^Ts\\ \widetilde V_m^Tt\end{bmatrix}.\label{tail}
\end{eqnarray}

From \eqref{correction}, the first elements of $\widetilde U_m^Ts$
and $\widetilde V_m^Tt$ are $\tilde u_1^Ts=0$ and $\tilde v_1^Tt=0$.
Denote $\widetilde U_m^{\prime}=[\tilde u_2,\dots,\tilde u_m]$,
$\widetilde V_m^{\prime}=[\tilde v_2,\dots,\tilde v_m]$ and
$\Theta_m^{\prime}=[\theta_2,\dots,\theta_m]$. Then \eqref{tail} becomes
\begin{equation*}
	r_{\mathrm{tail}}=\begin{bmatrix} & \!\!\!\! A\widetilde V_m^{\prime}-
		\widetilde U_m^{\prime}\Theta_m^{\prime} \\
		A^T\widetilde U_m^{\prime}
		-\widetilde V_m^{\prime}\Theta_m^{\prime}&\end{bmatrix}
	\begin{bmatrix} (\widetilde U_m^{\prime})^Ts\\
		(\widetilde V_m^{\prime})^Tt\end{bmatrix}
	=\begin{bmatrix} & \!\!\!\! R_1 \\  R_2&\end{bmatrix}
	\begin{bmatrix} (\widetilde U_m^{\prime})^Ts\\
		(\widetilde V_m^{\prime})^Tt\end{bmatrix},
\end{equation*}
where $R_1=A\widetilde V_m^{\prime}-\widetilde U_m^{\prime}\Theta_m^{\prime}$
and $R_2=A^T\widetilde U_m^{\prime}-\widetilde V_m^{\prime}\Theta_m^{\prime}$,
whose $(i-1)$-th columns are respectively the upper and lower parts of the
residual $r_i$
of $(\theta_{i},\tilde u_{i},\tilde v_{i})$, $i=2,\dots,m$; see \eqref{residual}.
Obviously, $(\widetilde U_m^{\prime})^Ts$ and $(\widetilde V_m^{\prime})^Tt$
are no larger than $s$ and $t$ in size.
By the standard perturbation theory \cite{golub2012matrix,parlett1998symmetric,stewart2001matrix}, it is known that
\begin{itemize}
\item $\|s\|$ and $\|t\|$ are of the same order as
the errors of $\tilde u_i$ and $\tilde v_i$, respectively;
\item $\|r_i\|$ is of the same order
as the errors of $\tilde u_{i},\tilde v_{i}$, $i=1,\dots,m$.
\end{itemize}
Therefore, whenever $\|r_i\|, i=2,\ldots,m$ are comparable to $\|r_1\|$,
one has $\|r_{\rm tail}\|=\mathcal{O}(\|r_1\|^2)$ and thus $r_1+r_{\rm tail}\approx r_1$
in \eqref{correction4} once \emph{$\|r_1\|$ becomes fairly small}.

Omitting the second order small term $r_{\mathrm{tail}}$ in
\eqref{correction4}, we have now obtained a new correction equation:
\begin{equation}\label{precorrection}
	\begin{bmatrix}	I-\widetilde U_m\widetilde U_m^T\!\!&
		\\&\!\!I-\widetilde V_m\widetilde V_m^T \end{bmatrix}
	\begin{bmatrix} -\tau I&A\\A^T&-\tau I \end{bmatrix}
	\begin{bmatrix}	I-\widetilde U_m\widetilde U_m^T\!\!&
		\\&\!\!I-\widetilde V_m\widetilde V_m^T \end{bmatrix}
	\begin{bmatrix} s\\ t\end{bmatrix} = -r_1.
\end{equation}
Its solution $(s,t)\perp\perp(\widetilde U_m,\widetilde V_m)$
is approximately equal to that of \eqref{correction} and
can thus expand  the current subspaces $\UU$ and $\VV$ as effectively as
the solution $(s,t)$ of \eqref{correction}.
The resulting method is the basic JDSVD-V.
With an approximate solution $(\tilde s^T,\tilde t^T)^T$ of
\eqref{precorrection} available, we orthonormalize $\tilde s$
and $\tilde t$ against $\widetilde U$ and $\widetilde V$
to obtain the subspace expansion vectors $u_+$ and $v_+$.
\begin{remark}
If $(\theta_1,\tilde{u}_1,\tilde{v}_1)$ is not yet a fairly
good approximation to $(\sigma_1,u_1,v_1)$, then $r_1$ is not small. Neither
is $r_{\rm tail}$. In this case, there is no theoretical winner between the solutions of
\eqref{correction4} and \eqref{precorrection}, and the two solutions
expand $\UU$ and $\VV$ equally effectively. We can argue why it is so:
The solution of each of the two correction equations amounts to one-step
inverse iteration for a {\em general} starting vector, so that their solutions
contribute to $u_1$ and $v_1$ approximately equally, i.e., $\UU$ and $\VV$ are expanded
equally effectively. Therefore,
since \eqref{correction} and \eqref{correction4} are equivalent,
we can always replace the standard correction equation \eqref{correction}
by the equally effective new \eqref{precorrection} no matter
whether or not $(\theta_1,\tilde{u}_1,\tilde{v}_1)$ is a fairly good approximation to
$(\sigma_1,u_1,v_1)$.
\end{remark}

\begin{remark}
Contrary to what might be thought, a JD method for the eigenvalue problem
converges linearly {\em rather than}
asymptotically quadratically for $\tau$ fixed or dynamically replaced by
the current approximate eigenvalue
as the method converges when solving
the standard and new correction equations approximately \cite[Sect. 2 of Chap. 6]{stewart2001matrix}.
This property straightforwardly applies to JDSVD. However,
it seems that such important property of a JD method has
received much less attention.
\end{remark}

\subsection{Convergence analysis on MINRES for solving \eqref{precorrection}}\label{subsec2}
In this section, we make a convergence analysis on MINRES
for the new correction equation \eqref{precorrection}.
Note that when $m=1$, \eqref{precorrection} degenerates to
\eqref{correction}. Therefore, we obtain the convergence results of MINRES for \eqref{correction} when
setting $m=1$.

Observe that the coefficient matrix in \eqref{precorrection}
is the restriction of
$B=\begin{bmatrix}\begin{smallmatrix}
		-\tau I&A\\A^T&-\tau I
\end{smallmatrix}\end{bmatrix}$
to the subspace
\begin{equation*}
	(\widetilde U_m,\widetilde V_m)^{\perp\perp}:=\{(x,y)\vert x
	\in\mathbb{R}^{M},y\in\mathbb{R}^N,
	x\perp\widetilde U_m, y\perp\widetilde V_m \},
\end{equation*}
and maps the subspace to itself.
Let $\widetilde P_m\in\mathbb{R}^{M\times(M-m)}$
and $\widetilde Q_m\in\mathbb{R}^{N\times(N-m)}$
be such that the matrices $[\widetilde U_m,\widetilde P_m]$ and
$[\widetilde V_m,\widetilde Q_m]$ are orthogonal. Then the columns of
\begin{equation}\label{WWm}
	\widetilde W_m=\begin{bmatrix}
		\widetilde P_m&\\&\widetilde Q_m
	\end{bmatrix}\in\mathbb{R}^{(M+N)\times(M+N-2)}
\end{equation}
form an orthonormal basis of $(\widetilde U_m,\widetilde V_m)^{\perp\perp}$,
and the orthogonal projector in \eqref{precorrection} is
\begin{equation*}
	\begin{bmatrix}
		I-\widetilde U_m\widetilde U_m^T&\\&I-\widetilde V_m\widetilde V_m^T
	\end{bmatrix}=\begin{bmatrix}
		\widetilde P_m\widetilde P_m^T&\\&\widetilde Q_m\widetilde Q_m^T
	\end{bmatrix}=\widetilde W_m\widetilde W_m^T.
\end{equation*}
Denote by $w=(s^T,t^T)^T$. Then the correction equation \eqref{precorrection} is
\begin{equation}\label{newform}
	\widetilde B_m w=-r_1 \qquad\mbox{with}\qquad w\in\mathcal{R}(W_m),
\end{equation}
where $\mathcal{R}(\cdot)$ denotes the column space of a matrix, and
\begin{equation}\label{Bm}
	\widetilde B_m=\widetilde W_m\widetilde B_m^{\prime}\widetilde W_m^T
	\quad\mbox{with}\quad
	\widetilde B_m^{\prime}=
\widetilde W_m^T	B\widetilde W_m
=\begin{bmatrix}-\tau I_{M-m}
		&\widetilde P_m^TA \widetilde Q_m
		\\ \widetilde Q_m^TA^T\widetilde P_m&-\tau I_{N-m}\end{bmatrix}.
\end{equation}
Let $\widetilde w_{j}=(\tilde s_j^T,\tilde t_j^T)^T$ be the
approximate solution of \eqref{newform} computed
by MINRES at the $j$th iteration,
$j=1,2,\ldots,M+N-2m$ with an initial guess $\widetilde w_0\in\mathcal{R}(W_m)$.
Then the corresponding residuals are
\begin{equation}\label{innerresidual}
	\qquad r_{\mathrm{in},j}=\widetilde B_m\widetilde w_j+r_1,\qquad\qquad
	j=0,1,\ldots,M+N-2m.
	\hspace{-6em}
\end{equation}

We now present an equivalent form of \eqref{precorrection} to be used later.

\begin{theorem}\label{thm6}
	With $\widetilde B_m^{\prime}\in \mathbb{R}^{(M+N-2m)\times (M+N-2m)}$ defined in \eqref{Bm},
	the solution to the correction equation
	\eqref{precorrection} is $w=\widetilde W_mz$, where $z=\widetilde
	W_m^Tw\in\mathbb{R}^{M+N-2m}$ solves the linear system
	\begin{equation}\label{precorrection2}
		\widetilde B_m^{\prime}z=-r^{\prime}_1
		\qquad\mbox{with}\qquad
		r^{\prime}_1=\widetilde W_m^Tr_1.
	\end{equation}
	If MINRES is applied to \eqref{precorrection2} with the
	initial guess $\tilde z_0=\widetilde
	W_m^Tw_0$, then the $j$th
	approximate solution $\tilde z_j=\widetilde W_m^T\widetilde w_j$
	with $\widetilde w_j=\widetilde W_m\tilde z_j$, $j=1,\ldots,M+N-2m$.
Denote by
$r_{\mathrm{in},j}^{\prime}=\widetilde B_m^{\prime} \tilde
	z_j+r^{\prime}_1$
the residual of $\tilde z_j$. Then
	\begin{equation}\label{resinorm2}
		\qquad\qquad \|r_{\mathrm{in},j}\|
		=\|r^{\prime}_{\mathrm{in},j}\|,\qquad j=1,2,\dots,M+N-2m.
	\end{equation}
\end{theorem}

\begin{proof}
	By \eqref{resperp} and \eqref{precorrection}, since
	$r_1\in(\widetilde U_m,\widetilde V_m)^{\perp\perp}$ and
	$w\in(\widetilde U_m,\widetilde V_m)^{\perp\perp}$, we have
	\begin{equation*}
		r_1=\widetilde W_m\widetilde W_m^Tr_1=\widetilde W_mr_1^{\prime}
		\qquad\mbox{and}\qquad
		w=\widetilde W_m\widetilde W_m^Tw=\widetilde W_mz
	\end{equation*}
	with $r^{\prime}_1=\widetilde W_m^Tr_1$ and $z=\widetilde W_m^Tw$.
	Substituting them and \eqref{Bm} into \eqref{newform}
	yields
	\begin{equation*}
		\widetilde W_m \widetilde B_m^{\prime}\widetilde W_m^T\widetilde W_mz
		=\widetilde W_m \widetilde B_m^{\prime}z
		=-\widetilde W_mr_1^{\prime},
	\end{equation*}
	premultiplying which by $\widetilde W_m^T$ leads to
	\eqref{precorrection2}.
	
	Write the $j$-dimensional Krylov subspace
	\begin{equation*}
		\mathcal{K}_j(\widetilde B_m,r_{\mathrm{in},0})
		=\mathrm{span}\{r_{\mathrm{in},0},\widetilde B_mr_{\mathrm{in},0},
		\widetilde B_m^2r_{\mathrm{in},0},\dots,\widetilde B_m^{j-1}r_{\mathrm{in},0}\}.
	\end{equation*}
	Since $\widetilde w_j\in\widetilde w_0+\mathcal{K}_j(\widetilde B_m,r_{\mathrm{in},0})$
	satisfies
	\begin{equation}\label{solution}
		\|r_{\mathrm{in},j}\|=\|\widetilde B_m\widetilde w_j+r_1\|=
		\min_{\widetilde w\in\widetilde w_0+\mathcal{K}_j(\widetilde B_m,r_{\mathrm{in},0})}
		\|\widetilde B_m\widetilde w+r_1\|,
	\end{equation}	
	by definition it is straightforward to verify from \eqref{Bm} and \eqref{innerresidual} that
	\begin{equation*}
		\mathcal{K}_j(\widetilde B_m,r_{\mathrm{in},0})=
		\mathcal{K}_j( \widetilde W_m\widetilde B_m^{\prime}\widetilde W_m^T,
		\widetilde W_mr^{\prime}_{\mathrm{in},0})
		=\widetilde W_m\mathcal{K}_j(\widetilde B_m^{\prime},r_{\mathrm{in},0}^{\prime});
	\end{equation*}
	that is, for any $\widetilde w\in \mathcal{K}_j(\widetilde B_m,r_{\mathrm{in},0})$,
	there exists $\tilde z\in\mathcal{K}_j(\widetilde B_m^{\prime},r_{\mathrm{in},0}^{\prime})$
	such that $\widetilde w=\widetilde W_m\tilde z$ and thus
	$\tilde z=\widetilde W_m^T\widetilde w$, $j=1,2,\ldots,M+N-2m$.
	Therefore, notice that the initial guesses for \eqref{newform} and \eqref{precorrection2} satisfies
$\widetilde w_0=\widetilde W_m\tilde z_0$.
From \eqref{solution} we obtain
\begin{eqnarray*}
\|r_{\mathrm{in},j}\| &=&
		\min_{\widetilde w\in\widetilde w_0+\mathcal{K}_j(\widetilde B_m,r_{\mathrm{in},0})}
		\|\widetilde W_m \widetilde B_m^{\prime}\widetilde W_m^T\widetilde w+\widetilde W_mr^{\prime}_1\|  \nonumber \\
&=&\min_{\tilde z\in \widetilde z_0+\mathcal{K}_j(\widetilde B_m^{\prime},r_{\mathrm{in},0}^{\prime})}
		\|\widetilde B_m^{\prime}\tilde z+r^{\prime}_1\|\nonumber \\
&=&\|r_{\mathrm{in},j}^{\prime}\|,
\end{eqnarray*}
	which proves \eqref{resinorm2} and $\widetilde
	w_j=\widetilde W_m\tilde z_j$.
\end{proof}

\cref{thm6} shows that the convergence of MINRES
for \eqref{precorrection} and \eqref{precorrection2} is identical.
Without loss of generality, take
$\widetilde w_0=\bm{0}$ and $\tilde z_0=\bm{0}$. Then
\begin{equation}\label{initial}
	r_{\mathrm{in},0}=r_1
	\qquad\mbox{and}\qquad
	r_{\mathrm{in},0}^{\prime}=r^{\prime}_1.
\end{equation}
Denote by $\tilde \lambda_i,i=1,\dots,M+N-2m$ the eigenvalues of
$\widetilde B_m^{\prime}$ in \eqref{Bm}.
It is known from \cite{greenbaum1997iterative}
that the residual norm of the approximate
solution $\tilde z_j$ of \eqref{precorrection2} satisfies
\begin{equation}\label{resminres}
	\frac{\|r_{\mathrm{in},j}^{\prime}\|}{\|r^{\prime}_1\|}\leq\min_{\pi\in\Pi_{j}}
	\max_{1\leq i\leq M+N-2m}|\pi(\tilde \lambda_i)|,
\end{equation}
where $\Pi_{j}$ denotes the set of polynomials of degree no more than $j$ and
$\pi(0)=1$ for $\pi\in\Pi_j$.
By \eqref{resinorm2} and \eqref{resminres}, we have the following estimates.

\begin{theorem}[{\cite[Chapter 3]{greenbaum1997iterative}}]\label{thm2}
	{\rm (\romannumeral 1)} If the eigenvalues $\tilde\lambda_i$  of
	$\widetilde B_m^{\prime}$ lie in the interval $[-\beta,-\alpha]$
	or $[\alpha,\beta]$ with $0<\alpha<\beta$, then
	\begin{equation}\label{convergesimplea}
		\frac{\|r_{\mathrm{in},j}\|}{\|r_1\|}\leq
		2\left(1-\tfrac{2}{1+\sqrt{\beta}/\sqrt{\alpha}}\right)^{j};
	\end{equation}
	{\rm (\romannumeral 2)} If the $\tilde\lambda_i\in [-\beta_1,-\alpha_1]\cup[\alpha_2,\beta_2]$ with
	$\alpha_1,\alpha_2,\beta_1,\beta_2>0$ and $\beta_2-\alpha_2=\beta_1-\alpha_1$, then
	\begin{equation}\label{convergesimple}
		\frac{\|r_{\mathrm{in},j}\|}{\|r_1\|}\leq 2\left(1-\frac{2}{1+ \sqrt{\beta_1\beta_2}/\sqrt{\alpha_1\alpha_2}}\right)^{\left\lbrack\frac{j}{2}\right\rbrack},
	\end{equation}
	where $\left\lbrack\frac{j}{2}\right\rbrack$ denotes the integer part of $\frac{j}{2}$.
\end{theorem}

This theorem indicates that the smaller $\beta-\alpha$ or $\beta_1-\alpha_1$ is,
the closer $\beta/\alpha$ or
$(\beta_1\beta_2)/(\alpha_1\alpha_2)$ is to one and thus the faster MINRES
converges. Bound \eqref{convergesimplea} is
the same as that for the CG method for solving
the same linear system \cite{greenbaum1997iterative}, where
the convergence factor is $\frac{\sqrt{\kappa(\widetilde B_m^{\prime})}-1}
{\sqrt{\kappa(\widetilde B_m^{\prime})}+1}$
with $\kappa(\widetilde B_m^{\prime})=\beta/\alpha$.
Suppose that $[-\beta_1,-\alpha_1]$ and
$[\alpha_2,\beta_2]$ are symmetric with respect to the origin, i.e.,
$\alpha_1=\alpha_2=\alpha$, $\beta_1=\beta_2=\beta$. Then the factor
in the right-hand side of \eqref{convergesimple} is
$\frac{\kappa(\widetilde B_m^{\prime})-1}{\kappa(\widetilde B_m^{\prime})+1}$,
the same as that for the steepest descent method applied to
a positive definite \eqref{precorrection2} with the
eigenvalue interval $[\alpha,\beta]$, but
the power is $\left\lbrack\frac{j}{2}\right\rbrack$ rather than $j$.
As a result, MINRES may converge much
more slowly for a highly indefinite linear system than it does
for a definite one.

Our next goal is to find the most compact interval or interval union
that contains all the eigenvalues $\tilde \lambda_i,i=1,\dots,M+N-2m$
of $\widetilde B_m^{\prime}$ in \eqref{Bm}, so that we are able to estimate the convergence rates
of MINRES as accurately as possible.

Before proceeding, we first look at the ideal case that
$\RR(\widetilde U_m)=\RR(U_m)$ and $\RR(\widetilde
V_m)=\RR(V_m)$. By \eqref{svdA}, we can take $\widetilde P_m= P_m$ and
$\widetilde Q_m=Q_m$ in \eqref{WWm} with
\begin{equation} \label{PPQQme}
	P_m=[u_{m+1},\dots,u_{M}]
	\qquad\mbox{and}\qquad
	Q_m=[v_{m+1},\dots,v_{N}].
\end{equation}
In this case, the matrix $\widetilde B_m^{\prime}$
defined in \eqref{Bm} reduces to the specially structured
\begin{equation}\label{BBmme}
	B_m^{\prime}=\begin{bmatrix} -\tau I_{M-m} &  P_m^TA  Q_m
		\\   Q_m^TA^TP_m & -\tau I_{N-m} \end{bmatrix}=
	\begin{bmatrix} -\tau I_{M-m} &  \Sigma_m^{\prime}
		\\    (\Sigma_m^{\prime})^T & -\tau I_{N-m} \end{bmatrix},	
\end{equation}
where
\begin{equation} \label{Sigmap}
	\Sigma_m^{\prime}= P_m^T A Q_m
	= \diag\{\sigma_{m+1},\dots,\sigma_{N}\}
	\in\mathbb{R}^{(M-m)\times(N-m)}.
\end{equation}
Therefore, the eigenvalues $\lambda_i$ of $B_m^{\prime}$ are
\begin{equation}\label{eigBm}
	\pm\sigma_{m+1}-\tau,\dots,\pm\sigma_{N}-\tau\qquad\mbox{and}\qquad
	\underbrace{-\tau, \dots,-\tau}_{M-N}.	
\end{equation}

In actual computations, however, we do not have $\RR(\widetilde U_m)=\RR(U_m)$ and $\RR(\widetilde
V_m)=\RR(V_m)$. Denote by
\begin{equation}\label{canonical}
	\Phi_m=\angle\left(\RR(\widetilde U_m), \RR(U_m)\right)
	\quad\mbox{and}\quad
	\Psi_m=\angle\left(\RR(\widetilde V_m), \RR(V_m)\right)
\end{equation}
the canonical angle matrices \cite[p.~329]{golub2012matrix} between $\RR(\widetilde U_m)$ and
$\RR(U_m)$ and that between $\RR(\widetilde V_m)$ and $\RR(V_m)$, respectively.
We next investigate the convergence of MINRES in detail. First of all,
we prove that $\widetilde B_m^{\prime}$ in \eqref{Bm} is orthogonally similar to
$B_m^{\prime}$ in \eqref{BBmme}
with the squared errors of $\RR(\widetilde U_m)$ and $\RR(\widetilde V_m)$, as shown below.

\begin{theorem}\label{thm8}
	$\widetilde B_m^{\prime}$ is orthogonally similar to $B_m^{\prime}$ with
	the error
	\begin{equation} \label{diffBBm}
		\|\widetilde B_m^{\prime}- Z^TB_m^{\prime}Z\|\leq
		\|A\|\left(\|\sin \Phi_m\|+\|\sin\Psi_m\| \right)^2,
	\end{equation}
	where $Z\in \mathbb{R}^{(M+N-2m)\times (M+N-2m)}$ is some orthogonal matrix; see
	\eqref{ZZm}.
\end{theorem}

\begin{proof} 		
	By \eqref{PPQQme}, let us decompose the orthonormal
	$\widetilde P_m$ and $\widetilde Q_m$, defined in the beginning
	of  \cref{subsec2}, as the orthogonal sums
	\begin{equation}\label{PPQQm}
		\widetilde P_m=P_mE_1+U_mE_2
		\qquad\mbox{and}\qquad
		\widetilde Q_m=Q_mF_1+V_mF_2
	\end{equation}
	with $E_1\in\mathbb{R}^{(M-m)\times(M-m)}$,
	$E_2\in\mathbb{R}^{m\times(M-m)}$,
	$F_1\in\mathbb{R}^{(N-m)\times(N-m)}$,
	$F_2\in\mathbb{R}^{m\times(N-m)}$.
	By the definition of sines of the canonical angles between two subspaces, we have
	\begin{equation}	 \label{sins2}
		\|\sin\Phi_m\|
		= \|\widetilde P_m^TU_m\| = \|E_2\|
		\quad\mbox{and}\quad
		\|\sin\Psi_m\|
		= \|\widetilde Q_m^TV_m\| = \|F_2\|.
	\end{equation}
	Let the SVDs of $E_1$ and $F_1$ be
	\begin{equation}\label{SVDEFm}
		E_1=X_1\Sigma_EX_2^T
		\qquad\mbox{and}\qquad
		F_1=Y_1\Sigma_FY_2^T,
	\end{equation}
where $X_1,X_2\in\mathbb{R}^{(M-m)\times(M-m)}$ and
$Y_1,Y_2\in\mathbb{R}^{(N-m)\times(N-m)}$ are orthogonal,
and the square diagonal matrices $\Sigma_E$ and $\Sigma_F$ satisfy
$\|\Sigma_E\|=\|E_1\|\leq 1$ and   $\|\Sigma_F\|=\|F_1\|\leq 1$.
Therefore, making use of \eqref{sins2}--\eqref{SVDEFm},
$E_1^TE_1+E_2^TE_2=I$ and $F_1^TF_1+F_2^TF_2=I$, we obtain
\begin{eqnarray}
	\qquad\quad  \|I-\Sigma_E\|&\leq&
	\|I-\Sigma_E^T\Sigma_E\|=\|I-E_1^TE_1\|=\|E_2^TE_2\|
	= \|E_2\|^2=\|\sin\Phi_m\|^2, \label{ICE}\\
	 \|I-\Sigma_F\|&\leq&\|I-\Sigma_F^T\Sigma_F\|=\|I-F_1^TF_1\|
	 =\|F_2^TF_2\|= \|F_2\|^2=\|\sin\Psi_m\|^2. \label{ICF}
\end{eqnarray}

	By the definitions of $P_m,Q_m$ and $U_m,V_m$, we have
	$U_m^TAV_m=\Sigma_{m}$, $P_m^TAQ_m=\Sigma_{m}^{\prime}$
	and $P_m^TAV_m=\bm{0}$, $U_m^TAQ_m=\bm{0}$, where $\Sigma_m$
	and $\Sigma_{m}^{\prime}$ are defined in \eqref{blocksvd} and
	\eqref{Sigmap}.
	Therefore, exploiting \eqref{PPQQm} and \eqref{SVDEFm}, we have
	\begin{equation}\label{PmAQm}
		\widetilde P_m^TA\widetilde Q_m =
		E_1^T\Sigma_m^{\prime}F_1+E_2\Sigma_mF_2
		= X_2\Sigma_EX_1^T\Sigma_m^{\prime}Y_1\Sigma_FY_2^T+E_2\Sigma_mF_2.
	\end{equation}
	Note that $X_i,Y_i,i=1,2$ are orthogonal and
	$\|\Sigma_m^{\prime}\|\leq\|A\|$, $\|\Sigma_m\|\leq\|A\|$.
	Exploiting \eqref{sins2} and \eqref{ICE}--\eqref{PmAQm}, we obtain
	\begin{eqnarray}
		\|\!\widetilde P_{m}^TA\widetilde Q_{m}\!-\!
		X_{\!2}X_{1}^T\Sigma_{m}^{\prime}Y_{\!1} Y_{\!2}^T\|
		&=&\|
		\Sigma_EX_1^T\Sigma_m^{\prime}Y_{1}\Sigma_F-
		X_1^T\!\Sigma_m^{\prime}Y_{\!1}
		+X_2^TE_2\Sigma_mF_2Y_2\| \nonumber\\
		&\leq&
		\|\!(\Sigma_{E}\!-\!I)X_1^T\Sigma_m^{\prime}Y_1\Sigma_{F}
		\!+\!X_1^T\Sigma_m^{\prime}Y_1(\Sigma_{F}\!-\!I)\!\|
		\!+\!\|\!E_{2}\Sigma_{m}F_{2}\!\|\nonumber\\
		&\leq&
		\|\Sigma_E-I\|\|\Sigma_m^{\prime}\|+
		\|\Sigma_m^{\prime}\|\|\Sigma_F-I\|+\|E_2\|\|\Sigma_m\|\|F_2\|   \nonumber\\
		&\leq &
		\|A\|(\|\sin\Phi_m\|^2+\|\sin\Psi_m\|^2+\|\sin\Phi_m\|\|\sin\Psi_m\|) \nonumber\\
		&\leq&
		\|A\|(\|\sin\Phi_m\|+\|\sin\Psi_m\|)^2.  \label{diffPAQ}
	\end{eqnarray}
	
	Denote the orthogonal matrix
	\begin{equation}\label{ZZm}
		Z=\begin{bmatrix}
			X_1X_2^T&\\&Y_1Y_2^T
		\end{bmatrix}\in\mathbb{R}^{(M+N-2m)\times(M+N-2m)}.
	\end{equation}
	By the definitions of $\widetilde B_m^{\prime}$ in \eqref{Bm}
	and $B_m^{\prime}$  in \eqref{BBmme}, we obtain
	\begin{equation*}
		\|\widetilde B_m^{\prime}-Z^T B_m^{\prime}Z\|=
		\|\widetilde P_m^TA\widetilde Q_m-X_2X_1^T\Sigma_m^{\prime}Y_1Y_2^T\|.
	\end{equation*}
	Then \eqref{diffBBm} follows from applying \eqref{diffPAQ} to the above.
\end{proof}

Since the matrix $Z$ in \eqref{ZZm} is orthogonal,
the eigenvalues $\lambda_i$ of $Z^TB_m^{\prime}Z$
are those of $B_m^{\prime}$, i.e., the ones in \eqref{eigBm}.
Label the eigenvalues $\lambda_i$ and $\tilde\lambda_i$ of $B_m^{\prime}$ and
$\widetilde B_m^{\prime}$ in the ascending or descending order; see \eqref{eigBm} for $\lambda_i$.
By the eigenvalue perturbation theory (cf. Theorem~10.3.1 of
\cite{parlett1998symmetric}), we obtain
\begin{equation}\label{diffeigs}
	|\tilde\lambda_i-\lambda_i|\leq
	\|\widetilde B_m^{\prime}- Z^TB_m^{\prime}Z\|.
\end{equation}
Therefore,
Theorem~\ref{thm8} shows that as long as $\RR(\widetilde U_m)$ and
$\RR(\widetilde V_m)$ are good approximations to
the left and right singular subspaces $\RR(U_m)$ and
$\RR(V_m)$ of $A$, the eigenvalues $\tilde{\lambda}_i$ of $\widetilde B_m^{\prime}$
are better approximations to the eigenvalues of
$B_m^{\prime}$ with the squared errors of $\RR(\widetilde U_m)$ and
$\RR(\widetilde V_m)$.

Combining Theorems~\ref{thm6}--\ref{thm2} with Theorem~\ref{thm8},
we present the following convergence
results on MINRES for \eqref{precorrection}.

\begin{theorem}\label{thm9}
	With the notations of \cref{thm8}, assume that
	\begin{equation}\label{delta2}
		\delta_m:=\|A\|(\|\sin\Phi_m\|+\|\sin\Psi_m\|)^2<|\sigma_{m+1}-\tau|.
	\end{equation}
	{\rm (\romannumeral 1)} If $\tau>(\sigma_{\max}+\sigma_{\max,m+1})/2$
	with $\sigma_{\max}$ and $\sigma_{\max,m+1}$ being the first and $(m+1)$th
	largest singular values of $A$, then 	
	\begin{equation}\label{converg6a}
		\frac{\|r_{\mathrm{in},j}\|}{\|r_1\|}\leq 2\Bigg(1-
		\frac{2}{1+\sqrt{(\tau+\sigma_{\max,m+1}+\delta_m)/
				(\tau-\sigma_{\max,m+1}-\delta_m)}}\Bigg)^{j}.
	\end{equation}
	{\rm (\romannumeral2)} If $\sigma_{\min}\leq\tau<(\sigma_{\min}+\sigma_{\min,m+1})/2$
	with $\sigma_{\min}$ and $\sigma_{\min,m+1}$ being the first and $(m+1)$th
	smallest singular values of $A$, and writing $j_o=\min\{1,M-N\}$, then
	\begin{equation}\label{converg7a}
		\frac{\|r_{\mathrm{in},j}\|}{\|r_1\|}\leq
		2\left(\frac{\sigma_{\max}+\delta_m}{\tau}\right)^{j_o}
		\Bigg(1-\frac{2}{1+\sqrt{\frac{(\sigma_{\max}+\delta_m)^2-\tau^2}
				{(\sigma_{\min,m+1}-\delta_m)^2-\tau^2}}}\Bigg)
		^{\left\lbrack\frac{j-j_o}{2}\right\rbrack}.
	\end{equation}
	{\rm (\romannumeral3)} If $(\sigma_{\min}+\sigma_{\min,m+1})/2
	<\tau<(\sigma_{\max}+\sigma_{\max,m+1})/2$, then
	\begin{equation}\label{converg8a}
		\frac{\|r_{\mathrm{in},j}\|}{\|r_1\|}\leq 2\left(1
		-\frac{2}{1+(\sigma_{\max}+\tau+\delta_m)/(|\sigma_{m+1}-\tau|-\delta_m)}\right)
		^{\left\lbrack\frac{j}{2}\right\rbrack}
	\end{equation}
with $\sigma_{m+1}$ being the $(m+1)$th closest singular value to $\tau$ in \eqref{svalues}.
\end{theorem}

\begin{proof}
	By \eqref{diffeigs}, it is known from \eqref{diffBBm} and \eqref{delta2} that the eigenvalues
	$\tilde\lambda_i$  and $\lambda_i$ of $\widetilde B_m^{\prime}$ in \eqref{Bm} and
	$B_m^{\prime}$ in \eqref{BBmme}   satisfy
	\begin{equation}\label{diffeigss}
		\qquad\qquad|\tilde\lambda_i-\lambda_i|\leq\delta_m,\qquad i=1,\dots,M+N-2m.
	\end{equation}
	
	(\romannumeral 1) In this case,
	from \eqref{svalues22} the $m$ singular values $\sigma_1$,
	$\dots$, $\sigma_{m}$  closest to $\tau$ are the $m$ largest ones
	$\sigma_{\max}$, $\sigma_{\max,2}$, $\dots$, $\sigma_{\max,m}$
	of $A$, and the remaining ones
	$\sigma_{m+1},\ldots, \sigma_N$ are smaller than $\tau$, among
	which $\sigma_{m+1}=\sigma_{\max,m+1}$.
	Therefore, the eigenvalues $\lambda_i$ of $B_m^{\prime}$ in
	\eqref{eigBm} are negative and lie in the interval
	\begin{equation*}
		[-(\sigma_{\max,m+1}+\tau),-(\tau-\sigma_{\max,m+1})].
	\end{equation*}
Then bound \eqref{diffeigss}
shows that the eigenvalues $\tilde\lambda_i$ of
$\widetilde B_m^{\prime}$ lie in
\begin{equation*}
	[-(\sigma_{\max,m+1}+\tau)-\delta_m,-(\tau-\sigma_{\max,m+1})+\delta_m].
\end{equation*}
Condition \eqref{delta2} indicates that $\tau-\sigma_{\max,m+1}>\delta_m$.
Then \eqref{converg6a} follows by taking $\alpha=\tau-\sigma_{\max,m+1}-\delta$ and
$\beta=\tau+\sigma_{\max,m+1}+\delta$ in \eqref{convergesimplea}.

(\romannumeral 2) In this case, the $m$ singular values $\sigma_1$,
	$\dots$, $\sigma_{m}$  closest to $\tau$ are the $m$ smallest ones
	$\sigma_{\min}\leq \sigma_{\min,2}\leq\cdots\leq \sigma_{\min,m}$ of $A$, and
	$\sigma_{m+1}=\sigma_{\min,m+1}$.
	The eigenvalues $\lambda_i$ except the $(M-N)$-multiple
	one $-\tau$ of $B_m^{\prime}$ in \eqref{eigBm}
	lie in the interval union
	\begin{equation*}
		 [-(\sigma_{\max}+\tau),
		-(\sigma_{\min, m+1}+\tau)]\cup[\sigma_{\min,m+1}-\tau,
		\sigma_{\max}-\tau].
	\end{equation*}
	 By definition \eqref{Bm}, $\widetilde B_m$ also
	has an $(M-N)$-multiple eigenvalue(s) $-\tau$, which disappears when $M=N$.
	Combining these with bound \eqref{diffeigss}, we know that the other $2(N-m)$
	eigenvalues $\tilde \lambda_i$ of $\widetilde B_m$
	lie in the union
	\begin{equation}\label{case2a}
	\Delta=	[-(\sigma_{\max}+\tau)-\delta_m,
		-(\sigma_{\min,m+1}+\tau)+\delta_m]\cup
		[\sigma_{\min,m+1}-\tau-\delta_m,\sigma_{\max}-\tau+\delta_m],
	\end{equation}
where $\sigma_{\min,m+1}-\tau = \sigma_{m+1}-\tau >\delta_m$ by \eqref{delta2}.

	 Construct a degree $j$ polynomial $q\in \Pi_{j}$ of form
	 \begin{equation*}
	 	q(\lambda)=\left(\tfrac{\tau+\lambda}{\tau}\right)^{j_o}\pi(\lambda)
	 	\qquad\mbox{with}\qquad
	 	\pi \in \Pi_{j-j_o},
	 \end{equation*}
	 where $j_o=\min\{1,M-N\}$. Then $q(0)=\pi(0)=1$, $q(-\tau)=0$ if $j_0=1$,
	 and $\left|\frac{\tau+\lambda}{\tau}\right|^{j_o}\leq
	 \left(\frac{\sigma_{\max}+\delta_m}{\tau}\right)^{j_o}$
	 over the eigenvalues $\tilde \lambda_i$ of $\widetilde B_m$; see \eqref{case2a}.
	 Therefore, by \cref{thm6}
	 and relations \eqref{initial}--\eqref{resminres}, we obtain
	 \begin{equation*}
	 	\frac{\|r_{\mathrm{in},j}\|}{\|r\|}\leq \min_{\pi\in \Pi_{j-j_o}}
	 	\max_{i=1,\dots,M+N-2m}|q(\tilde \lambda_i)|
	 	\leq \left(\frac{\sigma_{\max}+\delta_m}{\tau}\right)^{j_o}
	 	\min_{\pi\in  \Pi_{j-j_o}}
	 	\max_{\tilde\lambda_i\in \Delta}|\pi(\tilde \lambda_i)|,
	 \end{equation*}
	 Notice that the positive and negative intervals in \eqref{case2a} have the same
	 lengths $\sigma_{\max}-\sigma_{\min,m+1}+2\delta_m$.
	 Bound \eqref{converg7a} follows from applying \eqref{convergesimple} to the
	 above min-max problem.
	
	(\romannumeral 3) In this case, notice that the eigenvalues $\lambda_i$
	of $B_m^{\prime}$ in \eqref{eigBm} are in
	\begin{equation*}
		[-(\sigma_{\max}+\tau),-|\sigma_{m+1}-\tau|]\cup
		[|\sigma_{m+1}-\tau|,\sigma_{\max}+\tau].
	\end{equation*}
Then bound \eqref{diffeigs} means that the eigenvalues
$\tilde\lambda_i$ of $\widetilde B_m^{\prime}$ lie in
$$
	[-(\sigma_{\max}+\tau)-\delta_m,
	-|\sigma_{m+1}-\tau|+\delta_m]
	\cup[|\sigma_{m+1}-\tau|-\delta_m,\sigma_{\max}+\tau+\delta_m]
$$
with $|\sigma_{m+1}-\tau|-\delta_m>0$.
Since the two intervals above have the equal length $\sigma_{\max}+\tau-|\sigma_{m+1}-\tau|
+2\delta_m$,  applying \eqref{convergesimple} to MINRES for \eqref{precorrection}
with this interval union yields \eqref{converg8a}.
\end{proof}

\begin{remark}\label{rem1}
For the desired
$\sigma_1=\sigma_{\max},\sigma_{\min}$ and $\sigma_1\not=\sigma_{\max},\sigma_{\min}$,
the convergence rates of MINRES for the new correction equation \eqref{precorrection}
depend on
\begin{equation}\label{gamma123} \footnotesize{
		\gamma_{\mathrm{l},m}=\frac{\tau-\sigma_{\max,m+1}-\delta_m}{\tau+\sigma_{\max,m+1}+\delta_m}, \quad
		\gamma_{\mathrm{s},m}=\frac{(\sigma_{\min,m+1}-\delta_m)^2-\tau^2}{(\sigma_{\max}+\delta_m)^2-\tau^2}, \quad
		\gamma_{\mathrm{i},m}=\frac{|\sigma_{m+1}-\tau|-\delta_m}{\sigma_{\max}+\tau+\delta_m},}
	\end{equation}
respectively.
	The better $\sigma_{\max,m+1}$, $\sigma_{\min,m+1}$ or
	$\sigma_{m+1}$ is separated from $\tau$,
	 the smaller the convergence factors $1-\frac{2\sqrt{\gamma_{\mathrm{l},m}}}{1+\sqrt{\gamma_{\mathrm{l},m}}}$,
	$1-\frac{2\sqrt{\gamma_{\mathrm{s},m}}}{1+\sqrt{\gamma_{\mathrm{s},m}}}$ or
	$1-\frac{2\gamma_{\mathrm{i},m}}{1+\gamma_{\mathrm{i},m}}$
	in \eqref{converg6a}--\eqref{converg8a} are, so that MINRES converges faster for solving \eqref{precorrection}; in the meantime, the more accurate  $\mathcal{R}(\widetilde U_m)$
	and $\mathcal{R}(\widetilde V_m)$ approximate $\mathcal{R}(U_m)$ and $\mathcal{R}(V_m)$,
	i.e., the smaller $\delta_m$ defined by \eqref{delta2} is, the faster MINRES converges
for solving \eqref{precorrection}.
\end{remark}


\begin{remark}
	As mentioned at the beginning of this section, by taking $m=1$,
\eqref{gamma123} indicates that the convergence
	rates of MINRES for the standard correction equation \eqref{correction}
depend on
	\begin{equation}\label{gamma123s}
		\gamma_{\mathrm{l},1}=\frac{\tau-\sigma_{\max,2}-\delta_1}{\tau+\sigma_{\max,2}+\delta_1}, \quad
		\gamma_{\mathrm{s},1}=\frac{(\sigma_{\min,2}-\delta_1)^2-\tau^2}{(\sigma_{\max}+\delta_1)^2-\tau^2}, \quad
		\gamma_{\mathrm{i},1}=\frac{|\sigma_{2}-\tau|-\delta_1}{\sigma_{\max}+\tau+\delta_1},
	\end{equation}
respectively.
As is seen, if $\sigma_1\approx \tau$ and $\sigma_2$
is close to $\sigma_1$,
then MINRES may converge
very slowly for the standard correction equation \eqref{correction}.
Suppose that $\delta_1\approx \delta_m$ is sufficiently small,
i.e., $(\theta_i,\tilde u_i,\tilde v_i)$, $i=2,\dots,m$
are as accurate as $(\theta_i,\tilde u_1,\tilde v_1)$.
Under assumption \eqref{svalues22}, from \eqref{gamma123}--\eqref{gamma123s} we obtain
    $$
    \gamma_{\mathrm{l},1}\ll \gamma_{\mathrm{l},m},\qquad
    \gamma_{\mathrm{s},1}\ll \gamma_{\mathrm{s},m}, \qquad
    \gamma_{\mathrm{i},1}\ll \gamma_{\mathrm{i},m},
    $$
    meaning that  MINRES for the new correction equation \eqref{precorrection} will
    converge much faster than it does for the standard \eqref{correction} correspondingly.
As a result, if the desired $\sigma_1$ is close to $\sigma_2$ but well separated from $\sigma_{m+1}$,
the inner iterations of JDSVD-V can be much more efficient than those of the standard JDSVD.
\end{remark}

\subsection{Setup approach of the new correction equations in computations} \label{subsec3}

The previous results and analysis have shown that
it is much better to use MINRES to solve the new correction equation
\eqref{precorrection} rather than the standard \eqref{correction} when there are totally $m$
clustered singular values of $A$ closest to $\tau$ and their
approximate singular triplets have some accuracy.
Equally important is that the solutions of \eqref{precorrection} and \eqref{correction}
expand the subspaces equally effectively, meaning that the outer iterations of JDSVD-V
mimic those of JDSVD. In computations,
we need to set up necessary criteria to define $m$ mathematically
and determine it numerically, adaptively
form \eqref{precorrection}, and propose and develop a practical JDSVD-V algorithm.

Suppose that there are $m$ singular values $\sigma_i$ clustered at $\tau$ in
the sense of
\begin{equation*}
	|\sigma_i-\tau|
	\leq {\max\{\sigma_i,1\}} \cdot \tilde\varepsilon_1
\end{equation*}
for a reasonably small $\tilde\varepsilon_1$.
 \cref{thm8,thm9} indicate that if the singular vectors
with the $m$ clustered singular values are computed with some accuracy
\begin{equation*}
	\|\sin\Phi\|+\|\sin\Psi\|\leq\tilde\varepsilon_2
\end{equation*}
for a reasonably
small $\tilde\varepsilon_2$ then we benefit very much from MINRES for
solving \eqref{precorrection}.

At each outer iteration, for the current Ritz triplets
$(\theta_i,\tilde u_i,\tilde v_i)$, $i=2,\dots,k$,
if $\theta_i$ and the residual
$r_i=r(\theta_i,\tilde u_i,\tilde v_i)$ in
\eqref{residual} satisfy
\begin{equation}\label{close}
	|\theta_i-\tau|\leq \max\{\theta_i,1\}\cdot
	\tilde\varepsilon_1
	\qquad\mbox{and}\qquad
	\|r_i\|\leq \|A\|_{\mathrm{e}}\cdot \tilde\varepsilon_2,
\end{equation}
then we claim $\theta_i$ clustered at $\tau$ and accept
$(\theta_i,\tilde u_i,\tilde v_i)$ as a reasonably good approximate
singular triplet of $A$, where $\|A\|_{\mathrm{e}}=\sqrt{\|A\|_1\cdot\|A\|_{\infty}}$
with $\|A\|_1$ and $\|A\|_{\infty}$ the $1$- and $\infty$-norms of $A$, respectively.
Suppose that there are $\tilde m-1$ such clustered $\theta_i$,
and write them as $\theta_i,i=2,\ldots,\widetilde m$ for ease of notation.
Then we form
\begin{equation}\label{selected}
	\widetilde U_{\widetilde m}=[\tilde u_1,\tilde u_{2},\dots,\tilde u_{\widetilde m}]
	\qquad\mbox{and}\qquad
	\widetilde V_{\widetilde m}=[\tilde v_1,\tilde v_{2},\dots,\tilde v_{\widetilde m}],	
\end{equation}
and replace $\widetilde U_m$ and $\widetilde V_m$ by them in \eqref{precorrection}.

\begin{remark}\label{remark8}
\cref{thm9} indicates that provided that
$\delta_m<|\sigma_{m+1}-\tau|$ considerably, bounds \eqref{converg6a}--\eqref{converg8a}
are almost equal to the ideal ones with $\delta_m=0$.
In computations, it suffices to take a small fixed tolerance
$\tilde\varepsilon_2>\varepsilon_{\out}$ considerably. For instance, given an outer stopping
tolerance $\varepsilon_{\out}\in [10^{-14},10^{-8}]$, we have numerically found that
a good practical choice is
$\tilde\varepsilon_2 \in [10^{-3},10^{-2}]$.
\end{remark}

\begin{remark}\label{remark9}
Given $\tilde\varepsilon_2$, an appropriate value for
$\tilde\varepsilon_1$ is needed.
The larger $\tilde\varepsilon_1$ is, the larger $\tilde m$ is
in \eqref{selected}, and generally the fewer
inner iterations are required to solve \eqref{precorrection}.
On the other hand, as $\widetilde m$ increases, at each step of MINRES,
the matrix-vector product with the coefficient matrix
$\widetilde B_{\widetilde m}$ in \eqref{precorrection} costs more
when $A$ is sparse. We empirically
suggest to take $\tilde\varepsilon_1\in [0.05,0.01]$.
Once $\tilde\varepsilon_1$ is given, $m$ is determined, and
we will have $\tilde m=m$ ultimately as the method converges.
\end{remark}

\begin{remark}\label{remark10}
Two extreme cases are $\tilde\varepsilon_1=\tilde\varepsilon_2=0$ and
$\tilde\varepsilon_1=\tilde\varepsilon_2=+\infty$. In the first case,
$\widetilde m\equiv1$, and JDSVD-V degenerates to the standard JDSVD.
In the second case,  $\widetilde m=k$, and all $\tilde u_i$
and $\tilde v_i$, $i=2,\dots,k$ enter into the
new correction equation \eqref{precorrection}, and
the resulting JDSVD-V solves the correction equation \eqref{correctionV},
a specific instance of the JD method proposed
in \cite{genseberger1999alternative} for the eigenvalue problem, denoted by JDSVD-V($+\infty$) in
the sequel. Such JD method had been argued to work poorly \cite{deSturler2002ImprovingTC,genseberger1999alternative}
and will be further confirmed numerically in this paper. As has been clear from
the current context, the fundamental
reason is that the {\em whole} $\mathcal{R}(\tilde{U}_k)$ and $\mathcal{R}(\tilde{V}_k)$
are generally {\em not} reasonably good approximate left and right singular subspaces of $A$,
so that the corresponding term $r_{\rm tail}$ overwhelms $r_1$ in \eqref{correction4} and
the solution of the corresponding correction equation \eqref{precorrection} is far from
that of \eqref{correction}. As a consequence, such solution is ineffective to expand
the subspaces, and the corresponding JDSVD-V($+\infty$) method consumes more outer iterations to
converge than the JDSVD and JDSVD-V methods with the correction equation \eqref{correction}
and \eqref{precorrection}, respectively, do. Our reasoning applies to that JD method in \cite{genseberger1999alternative} for the eigenvalue problem,
\end{remark}

\begin{remark}
For $k$ small, there may be no sufficiently
good approximate singular triplet that satisfies
\eqref{close}. In this case, $\widetilde m=1$ in \eqref{selected},
and \eqref{precorrection} degenerates to \eqref{correction}.
However, for $m>1$, as the dimension $k$ increases,
more and more Ritz approximations start to converge and
fulfill \eqref{close}, so that $\widetilde m>1$
pairs of Ritz vectors participate in \eqref{precorrection}.
Then JDSVD-V starts to exhibit its superiority of overall efficiency to JDSVD.
\end{remark}

\section{A thick-restart JDSVD-V algorithm with deflation and purgation}\label{sec:5}
For practical purpose, we must limit the subspace dimension $k\leq k_{\max}$.
If JDSVD and JDSVD-V do not yet converge until $k=k_{\max}$, we must restart them.
Furthermore, if we are required to compute $\ell>1$ singular triplets of $A$ corresponding to
the $\ell$ singular values $\sigma_1,\ldots,\sigma_{\ell}$ closest to
$\tau$ (cf. \eqref{svalues}), an appropriate deflation
is necessary. Meanwhile, it will turn out that a certain purgation is crucial and
may greatly reduce the outer iterations
for computing the second to the $\ell$th singular triplets.

\subsection{A new thick-restart} \label{subsec4}
If JDSVD-V does not yet converge for computing the desired
$\ell$ singular triplets for a given $k_{\max}$, we modify a commonly used thick restart,
initially proposed in \cite{stath1998} and later popularized in the JDSVD
and JDGSVD type methods \cite{huang2019inner,huang2022harmonic,huang2023cross,
	wu2015preconditioned}, and
propose a more effective thick-restart
JDSVD-V than the standard thick-restart JDSVD. Our thick-restart retains
some $k_{\new}$-dimensional restarting subspaces with
\begin{equation}\label{knew}
	k_{\new}=\max\{k_{\min},\widetilde m\}	
\end{equation}
where $k_{\min}> 1$ is a user-prescribed dimension of restarting
subspaces in the standard thick-restart JDSVD
and $\widetilde m$ is the number of Ritz approximations
that participate in the correction equation~\eqref{precorrection} in
the current expansion phase.

Specifically, assume that $\widetilde m$ Ritz approximations
$(\theta_i,\tilde u_i,\tilde v_i),i=1,2,\dots,\widetilde m$ of $A$
participate in \eqref{precorrection} before restart.
If $\widetilde m \leq k_{\min}$, we simply use the standard thick-restart scheme,
generate the orthonormal basis matrices of the new $k_{\min}$-dimensional subspaces
$\UU$ and $\VV$ by
\begin{equation*}
	\widetilde U=[\tilde u_1,\tilde u_2,\dots,\tilde u_{k_{\min}}],
	\qquad
	\widetilde V=[\tilde v_1,\tilde v_2,\dots,\tilde v_{k_{\min}}],
\end{equation*}
and updates the projection matrix of $A$ with respect to these two subspaces by
\begin{equation*}
	H=\widetilde{U}^TA\widetilde{V}=\diag\{\theta_1,\theta_{2},\dots,\theta_{k_{\min}}\}.
\end{equation*}
If $\widetilde m> k_{\min}$, we form the orthonormal
basis matrices of the new $\widetilde m$-dimensional
$\UU$ and $\VV$ and the corresponding projection matrix of $A$ by
\begin{equation*}
	\widetilde U=\widetilde U_{\widetilde m},
	\qquad
	\widetilde V=\widetilde V_{\widetilde m},
	\qquad
	H=\widetilde{U}_{\widetilde m}^TA\widetilde{V}_{\widetilde m}=
	\diag\{\theta_1,\theta_2,\dots,\theta_{\widetilde m}\}
\end{equation*}
with $\widetilde U_{\widetilde m}$ and $\widetilde V_{\widetilde m}$ defined by \eqref{selected}.
In this way,
the $k_{\rm new}$-dimensional restarting subspaces contain approximate left and
right singular vectors associated with all available approximate singular values clustered at
$\tau$. Consequently,
whenever $m<k_{\max}$ considerably, the new thick-restart JDSVD-V is expected to
substantially
outperform the standard thick-restart JDSVD in terms of the outer iterations.

\subsection{Deflation and purgation}\label{subsec5}
Suppose that the $\ell(>1)$ singular triplets of $A$ are of interest with
the $\ell$ singular values $\sigma_i$ closest to the target $\tau$.
Then JDSVD-V equipped with the deflation in \cite{huang2019inner}
meets this demand.
Specifically, assume that $j<\ell$ approximate singular triplets
$\left(\theta_{i,\mathrm{c}},u_{i,\mathrm{c}}, v_{i,\mathrm{c}}\right)$
have already converged to $(\sigma_i,u_i,v_i)$,
$i=1,\dots,j$ with the residual $r_{i,\mathrm{c}}
=r(\theta_{i,\mathrm{c}},u_{i,\mathrm{c}}, v_{i,\mathrm{c}})$ defined
by \eqref{residual} satisfying
\begin{equation*}
	\|r_{i,\mathrm{c}}\|\leq \|A\|_{\mathrm{e}}\cdot \varepsilon_{\mathrm{out}},\qquad i=1,\dots,j,
\end{equation*}
where $\|A\|_{\rm e}$ is defined as in \eqref{close}.
Denote
\begin{equation*}
	\Sigma_{\mathrm{c}}=\mathrm{diag}\{\theta_{1,\mathrm{c}},\dots,
	\theta_{j,\mathrm{c}}\},\qquad
	U_{\mathrm{c}}=[u_{1,\mathrm{c}},\dots,u_{j,\mathrm{c}}],
	\qquad
	V_{\mathrm{c}}=[v_{1,\mathrm{c}},\dots,v_{j,\mathrm{c}}].
\end{equation*}
Then $\left(\Sigma_{\mathrm{c}},U_{\mathrm{c}}, V_{\mathrm{c}}\right)$ is a converged
approximation to the partial SVD $(\Sigma_j,U_j,V_j)$ of $A$ in \eqref{blocksvd}.
Suppose that the next singular triplet $(\sigma_{j+1},u_{j+1},v_{j+1})$ of $A$
is desired. At the extraction phase of current context,
the deflation in \cite{huang2019inner} implements
JDSVD-V with respect to the $k$-dimensional $\UU$ and $\VV$
that {\em are orthogonal to} $\RR(U_{\mathrm{c}})$ and $\RR(V_{\mathrm{c}})$,
respectively. We obtain approximate
singular triplets $(\theta_i,\tilde u_i,\tilde v_i)$, $i=1,\dots,k$ as
described in \cref{sec:2}, and take $(\theta_1,\tilde u_1,\tilde v_1)$
as an approximation to $(\sigma_{j+1},u_{j+1},v_{j+1})$.

If $(\theta_1,\tilde u_1,\tilde v_1)$ does not converge, we take
$(\theta_i,\tilde u_i,\tilde v_i)$, $i=1,2,\dots,\widetilde m$
that meet \eqref{close} to form $\widetilde U_{\widetilde m}$ and $\widetilde V_{\widetilde m}$
in \eqref{selected}.
Then we set $U_{\mathrm{p}}=[U_{\mathrm{c}}, \widetilde U_{\widetilde m}]$ and
$V_{\mathrm{p}}=[ V_{\mathrm{c}}, \widetilde V_{\widetilde m}]$, and
use MINRES to approximately solve the new correction equation
\begin{equation}\label{deflatcorrection}
	\begin{bmatrix}	I-U_{\mathrm{p}}U_{\mathrm{p}}^T\!\! &\\&\!\!I- V_{\mathrm{p}}V_{\mathrm{p}}^T \end{bmatrix}
	\begin{bmatrix} -\tau I&A\\A^T&-\tau I \end{bmatrix}
	\begin{bmatrix}	I-U_{\mathrm{p}}U_{\mathrm{p}}^T\!\!&\\&\!\!I-V_{\mathrm{p}}  V_{\mathrm{p}}^T \end{bmatrix}
	\begin{bmatrix} s\\ t\end{bmatrix}=r_{\mathrm{p}}
\end{equation}
for $(s,t)\perp\perp(U_{\mathrm{p}},V_{\mathrm{p}})$,
where
$r_{\mathrm{p}}= -\begin{bmatrix}	I-U_{\mathrm{c}}U_{\mathrm{c}}^T\!\!\!
	&\\&\!\!\! I- V_{\mathrm{c}}V_{\mathrm{c}}^T \end{bmatrix}r_1$
with
$r_1=r(\theta_1,\tilde u_1,\tilde v_1)$ being
the residual of $(\theta_1,\tilde u_1,\tilde v_1)$ defined by \eqref{residual}.
We compute an approximate solution $(\tilde s,\tilde t)\perp\perp(U_{\mathrm{p}},V_{\mathrm{p}})$
and terminate MINRES when the residual $r_{\mathrm{in}}$ in \eqref{innerresidual} satisfies
\begin{equation}\label{innerconvergence}
	\|r_{\mathrm{in}}\|\leq  \|r\| \cdot  \min\{{\omega \varepsilon_{\mathrm{in}}},0.1\},
\end{equation}
where $\varepsilon_{\mathrm{in}}\in[10^{-4},10^{-3}]$ is a user-prescribed
tolerance and $\omega$ is a constant
that is estimated by using all the approximate singular values
computed at the current cycle; see \cite{huang2019inner,jia2014inner}
for theoretical details on $\varepsilon_{\mathrm{in}}$.
We use the approximate solution $\tilde s$ and $\tilde t$ to expand
$\UU$ and $\VV$. Since $(\widetilde U,\widetilde
V)\perp\perp(U_{\mathrm{c}},V_{\mathrm{c}})$ and
$(\tilde s,\tilde t)\perp\perp(U_{\mathrm{c}},V_{\mathrm{c}})$, the expanded
$\UU$ and $\VV$ are automatically orthogonal to $\RR(U_{\mathrm{c}})$
and $\RR(V_{\mathrm{c}})$, respectively.
The thick-restart JDSVD-V proceeds until $(\theta_1,\tilde u_1,\tilde v_1)$ converges.
We then assign $\left(\theta_{j+1,\mathrm{c}},u_{j+1,\mathrm{c}},
v_{j+1,\mathrm{c}}\right):=(\theta_1,\tilde u_1,\tilde v_1)$, add it to the
converged $\left(\Sigma_{\mathrm{c}},  U_{\mathrm{c}},V_{\mathrm{c}}\right)$,
and set $j=j+1$. Proceed in such a way until all the $\ell$ singular triplets are found.


Regarding the purgation issue, notice that the current $k$-dimensional left and right
subspaces $\UU$ and $\VV$ contain rich information on the next
desired $(u_{j+1},v_{j+1})$ when we have computed $(u_j,v_j)$.
Therefore, when computing $(\sigma_{j+1},u_{j+1},v_{j+1})$,
one can find a better initial approximation
to $(u_{j+1},v_{j+1})$ by fully exploiting the current $\UU$ and $\VV$ rather than the ones
generated randomly. The purgation approach \cite{huang2019inner}
purges or purifies the converged $\tilde u_j,\tilde v_j$ from the current
subspaces $\UU$ and $\VV$ with little cost, and uses the purged
$(k-1)$-dimensional as the initial restarting subspaces. One can
benefit much from using such purged lower-dimensional subspaces and save outer iterations
considerably when computing
$(\sigma_i,u_i,v_i),\ i=2,\ldots,\ell$, compared to the computation of
the first desired $(\sigma_1,u_1,v_1)$. The approach  in \cite{huang2019inner}
is directly applicable to the thick-restart JDSVD-V, and we thus omit the details.

\begin{algorithm}[htbp]
	\caption{Thick-restart JDSVD-V
		with deflation and purgation.}
	\begin{algorithmic}[1]\label{algo1}
		\STATE{\textbf{Initialization:}\  Given the target $\tau$, set $k=1$, $k_{\mathrm{c}}=0$,
			$\Sigma_{\mathrm{c}}=[\ ], U_{\mathrm{c}}=[\ ]$, and $V_{\mathrm{c}}=[\ ]$.
			Let $\widetilde U=[\  ]$, $\widetilde V=[\ ]$, $H=[\ ]$,
			and choose normalized
			starting vectors $u_{+}=\frac{u_0}{\|u_0\|}$ and $v_{+}=\frac{v_0}{\|v_0\|}$.\hspace{-2em}}
		\STATE{Calculate $\|A\|_e=\sqrt{\|A\|_1\cdot\|A\|_{\infty}}$.}
		
		\WHILE{$k\geq0$}
		
		\STATE{\label{step:3} Set $\widetilde U=[\widetilde U, u_{+}]$ and $\widetilde V=[\widetilde V, v_{+}]$, and update
			$H=\widetilde U^TA\widetilde V$.}
		
		\STATE{\label{step4} Compute the singular triplets $(\theta_i,c_i,d_i),i=1,\dots,k$ of $H$ with $\theta_i$  ordered increasingly by their distances from $\tau$. Form the approximate singular triplets $(\theta_i,\tilde u_i\!=\widetilde Uc_i,\tilde v_i\!=\widetilde Vd_i)$ and the residuals $r_i\!=\!\bsmallmatrix{A\tilde v_i-\theta_i \tilde u_i\\ A^T\tilde u_i-\theta_i \tilde v_i}$, $i\!=\!1,\dots,k$.}
		
		\IF{$\|r_1\|\leq\|A\|_{\mathrm{e}}\cdot \varepsilon_{\mathrm{out}}$}
		\STATE{Update $\Sigma_{\mathrm{c}}=\diag\{\Sigma_{\mathrm{c}},\theta_1\},
			U_{\mathrm{c}}=[U_{\mathrm{c}},\tilde u_1]$, $V_{\mathrm{c}}=[V_{\mathrm{c}},\tilde v_1]$,
			and set $k_c=k_c+1$.}
		
		\STATE{\textbf{if} $k_{\mathrm{c}}=\ell$ \textbf{then} return
			$(\Sigma_{\mathrm{c}},U_{\mathrm{c}},V_{\mathrm{c}})$ and stop. \textbf{fi}}
		
		\STATE{\label{step5} Perform the purgation approach  \cite{huang2019inner},
set $k=k-1$, and go to step~\ref{step4}.}
		\ENDIF
		
		\STATE{Determine the $\widetilde m-1$ Ritz approximations that satisfy \eqref{close},
			and form $\widetilde U_{\widetilde m}=[\tilde u_1,\tilde u_2,\dots,\tilde u_{\widetilde m}]$
			and $\widetilde V_{\widetilde m}=[\tilde v_1,\tilde v_2,\dots,\tilde u_{\widetilde m}]$.}
		
		\STATE{Set $U_{\mathrm{p}}=[U_{\mathrm{c}},\widetilde U_{\widetilde m}]$,
			$V_{\mathrm{p}}=[V_{\mathrm{c}},\widetilde V_{\widetilde m}]$ and
			$r_{\mathrm{p}}=-\bsmallmatrix{I-U_{\mathrm{c}}U_{\mathrm{c}}^T&\\&I- V_{\mathrm{c}}V_{\mathrm{c}}^T}r_1$, and solve
			$$\begin{bmatrix}	I-U_{\mathrm{p}}U_{\mathrm{p}}^T &\\&I- V_{\mathrm{p}}V_{\mathrm{p}}^T \end{bmatrix}
			\begin{bmatrix} -\tau I&A\\A^T&-\tau I \end{bmatrix}
			\begin{bmatrix}	I-U_{\mathrm{p}}U_{\mathrm{p}}^T&\\&I-V_{\mathrm{p}}  V_{\mathrm{p}}^T \end{bmatrix}
			\begin{bmatrix} s\\ t\end{bmatrix} = r_{\mathrm{p}}$$ 	
			for the approximate solution $(\tilde s,\tilde t) \perp\perp(U_{\mathrm{p}},V_{\mathrm{p}})$
			with the residual norm $\|r_{\mathrm{in}}\|$ satisfying the stopping criterion \eqref{innerconvergence}.
		}
		
		\STATE{\label{step:12}\textbf{if} $k=k_{\max}$ \textbf{then} perform the thick-restart and reset $k=k_{\new}$ in \eqref{knew}. \textbf{fi}}
		
		\STATE{\label{step:13} Orthonormalize $\tilde s$ and $\tilde t$ against $\widetilde U$ and $\widetilde V$
			to obtain the expansion vectors $u_{+}$ and $ v_{+}$, respectively, and set $k=k+1$.}
		\ENDWHILE
	\end{algorithmic}
\end{algorithm}

\subsection{Thick-restart JDSVD-V algorithm with deflation and purgation} \label{subsec6}
 \cref{algo1} sketches our thick-restart JDSVD-V
with deflation and purgation.
It requires 
the target $\tau$, the number $\ell$ of desired singular
triplets, and the outer stopping tolerance $\varepsilon_{\mathrm{out}}$, and
outputs a converged partial SVD $(\Sigma_c,U_c,V_c)$ that satisfies
\begin{equation*}
	\sqrt{\|AV_c-U_c\Sigma_c\|_{\mathrm{F}}^2+\|A^TU_c-V_c\Sigma_c\|_{\mathrm{F}}^2}\leq
	\sqrt{\ell}\|A\|_{\mathrm{e}}\cdot \varepsilon_{\mathrm{out}},
\end{equation*}
where $\|\cdot\|_{\rm F}$ is the F-norm of a matrix.
Other optional parameters are starting vectors
$u_0\in\mathbb{R}^{M}$ and $v_0\in\mathbb{R}^{N}$,
the maximum and minimum dimensions $k_{\max}$ and $k_{\min}$ of the searching subspaces,
the accuracy requirement $\varepsilon_{\mathrm{in}}$ for expansion vectors in \eqref{innerconvergence},
and the parameters $\tilde\varepsilon_1$ and $\tilde\varepsilon_2$ in \eqref{close}.
By defaults, we set $\varepsilon_{\mathrm{out}}=10^{-12}$, $u_0={\sf randn}(M,1)$,
$v_0={\sf randn}(N,1)$,
$k_{\max}=30$, $k_{\min}=3$,
$\varepsilon_{\mathrm{in}}=10^{-3}$,
$\tilde\varepsilon_1=0.05$ and $\tilde\varepsilon_2=0.01$, where {\sf randn} is the
{\sc Matlab} built-in function
that generates normally distributed random matrices or vectors.


\section{Numerical experiments}\label{sec:6}
We report numerical experiments to illustrate the considerable
superiority of JDSVD-V to JDSVD.
We also make a comparison of JDSVD-V and PRIMME\_SVDS \cite{wu16primme}. 
To this end, we have also developed a hybrid two-stage JDSVD-V algorithm, called
JDSVD-V\_HYBRID,
in {\sc Matlab}, which is based on the same motivation and is along the same line as that of
the state-of-the-art two-stage software package PRIMME\_SVDS in C language
developed in \cite{wu16primme}
for accurately computing several extreme singular triplets.
In the first stage of JDSVD-V\_HYBRID we
perform a similar JDSVD-V method on the eigenproblem
of $A^TA$ or that of $AA^T$ if $m<n$ with
the same outer stopping tolerance as that of PRIMME\_SVDS,
where $A$ has full column or row  rank; a convergence analysis
of JDSVD-V on such eigenproblem can be adaptable
from the current context, and we omit it due to space.
If further accuracy is required,
JDSVD-V\_HYBRID then switches to the second stage and implements the JDSVD-V method proposed in
this paper using the approximate singular vectors obtained in the first stage
as initial guesses.
We will illustrate that JDSVD-V\_HYBRID is at least competitive with
and can outperform PRIMME\_SVDS considerably in terms of overall efficiency
and reliability, especially when computing the smallest singular triplets.

All the experiments were implemented on a 12th Gen Intel(R)
Core(TM) i7-12700K CPU 3.60GHz with 32 GB main memory and
12 cores using the {\sc Matlab} R2023b under the Ubuntu 22.04.4
LTS 64 bit system with the machine precision $\epsilon=2.22\times 10^{-16}$.

\begin{table}[tbhp]
	\caption{Properties of the test matrices from the
		SuiteSparse Matrix Collection \cite{davis2011university}: part I.}\label{table00}
	\begin{center}
		\begin{tabular}{cccccc} \toprule
			$A$&$M$&$N$&$\textit{nnz}$&$\sigma_{\max}$&$\sigma_{\min}$  \\ \midrule
			e40r0100 &17281 &17281&553562&13.1&8.68e-8 \\
			inlet &11730 &11730&328323&6.57&1.69e-6 \\
			Alemdar &6245 &6245 &42581 &69.5 &3.30e-3 \\
			bcsstk33 &8738 &8738 &591904 &75.8 &2.18e-95 \\
			cat\_ears\_3\_4 &5226 &13271 &39592 &5.49 &3.06e-1 \\
			epb2 &25228 &25228 &175027 &3.34 &1.28e-3 \\			
			relat8 & 345688 & 12347 &1334038& 18.8 &0 \\
			dw8192 &8192 &8192 &41746 &110 &2.88e-5 \\
			garon2 &13535 &13535 &373235 &44.9 &7.32e-7 \\
			tomographic1 &73159 &59498 &647495 &6.98 &0 \\
			\bottomrule
		\end{tabular}
	\end{center}
\end{table}

\cref{table00} lists some of the test matrices from the
SuiteSparse Matrix Collection \cite{davis2011university} together
with some of their basic properties, where $\textit{nnz}$ denotes
the total number of nonzero entries in $A$, and the largest and smallest
singular values $\sigma_{\max}$ and $\sigma_{\min}$ of $A$ are,
only for experimental purpose, computed by the {\sc Matlab} functions
{\sf svds} and {\sf svd} for the very large matrices ``relat8'',
``tomographic1'' and all the other moderately large ones, respectively.
To illustrate the wide applicability of our algorithm,
we took $\tau$ at will, so that they correspond to either extreme or interior
singular triplets.
For each matrix with the given target $\tau$,
we took ten different pairs of random starting vectors
$u_0={\sf randn}(M,1)$ and $v_0={\sf randn}(N,1)$, and
ran each algorithm ten times with all the other default
parameters as described in \cref{subsec6} unless specified otherwise.
For \eqref{correction}, \eqref{precorrection} and \eqref{deflatcorrection},
we took zero vector as the initial guesses and used the {\sc Matlab} function {\sf minres}
to solve them until the relative residual norms met \eqref{innerconvergence}.
We report the total number of outer iterations ``Iter'',
the matrix-vector products ``MVs'' with $A$ and $A^T$, and
the CPU time (``Time'') that each algorithm
had averagely consumed over ten runs.

\begin{exper}
	By taking $\tilde\varepsilon_1=\tilde\varepsilon_2=0$ and
	$\tilde\varepsilon_1=0.05$, $\tilde\varepsilon_2=0.01$ and
	$\tilde\varepsilon_1=\tilde\varepsilon_2=+\infty$,
	we use the standard JDSVD, JDSVD-V and JDSVD-V($+\infty$)
	(cf. \cref{remark10}) algorithms to compute the ten singular
	triplets of $A=$ ``e40r0100'' and ``inlet'' for the targets
	$\tau = 13,0.1$ and $5.3,0.3$, respectively.
\end{exper}

\begin{figure}[tbhp]
	\centering
	\includegraphics[width=0.48\textwidth]{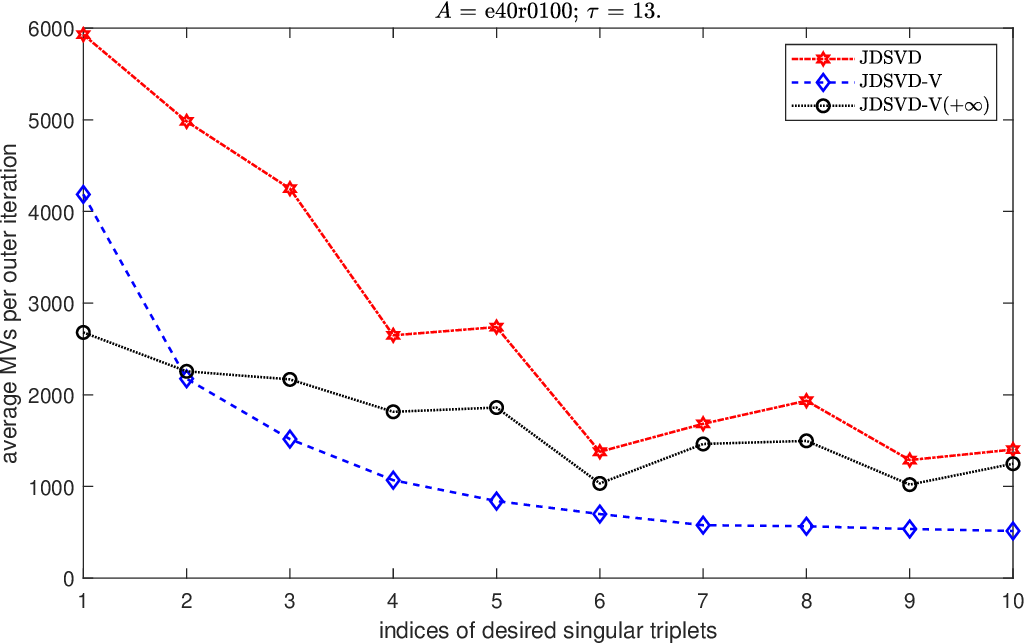}\hfill
	\includegraphics[width=0.47\textwidth]{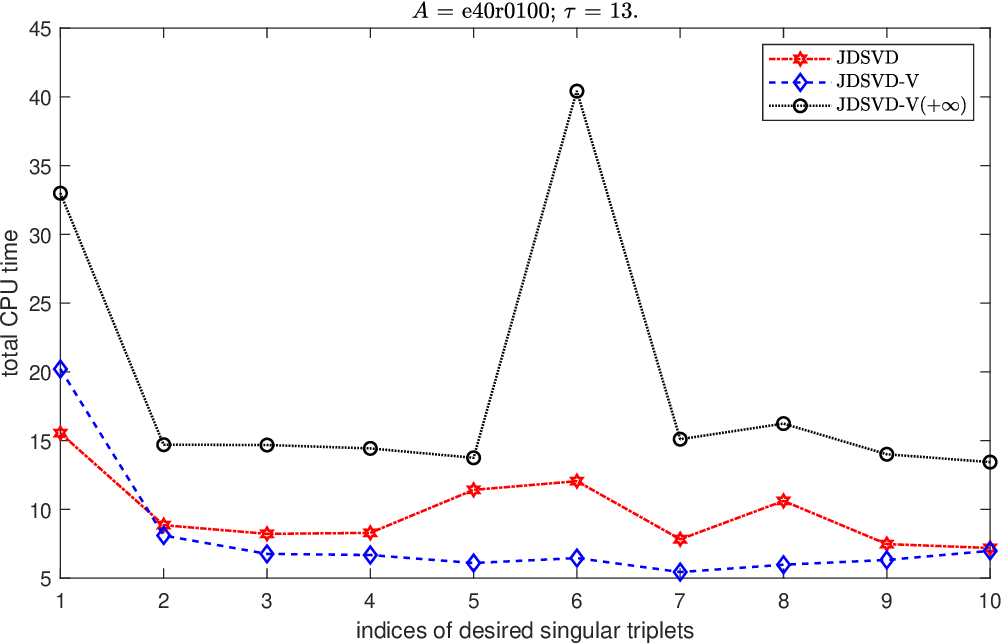}\\[0.2em]
	
	\includegraphics[width=0.48\textwidth]{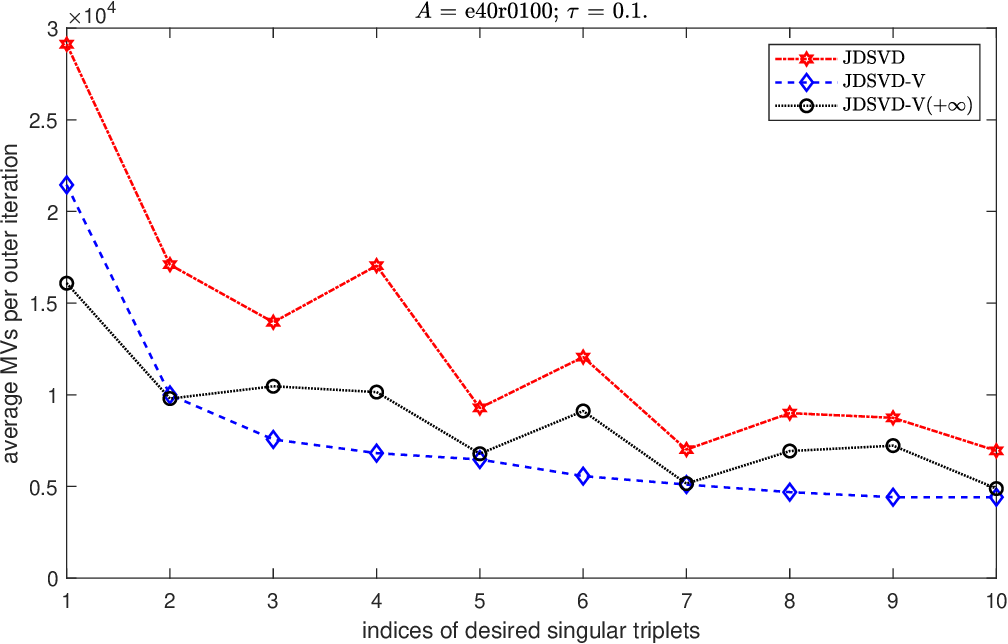}\hfill
	\includegraphics[width=0.48\textwidth]{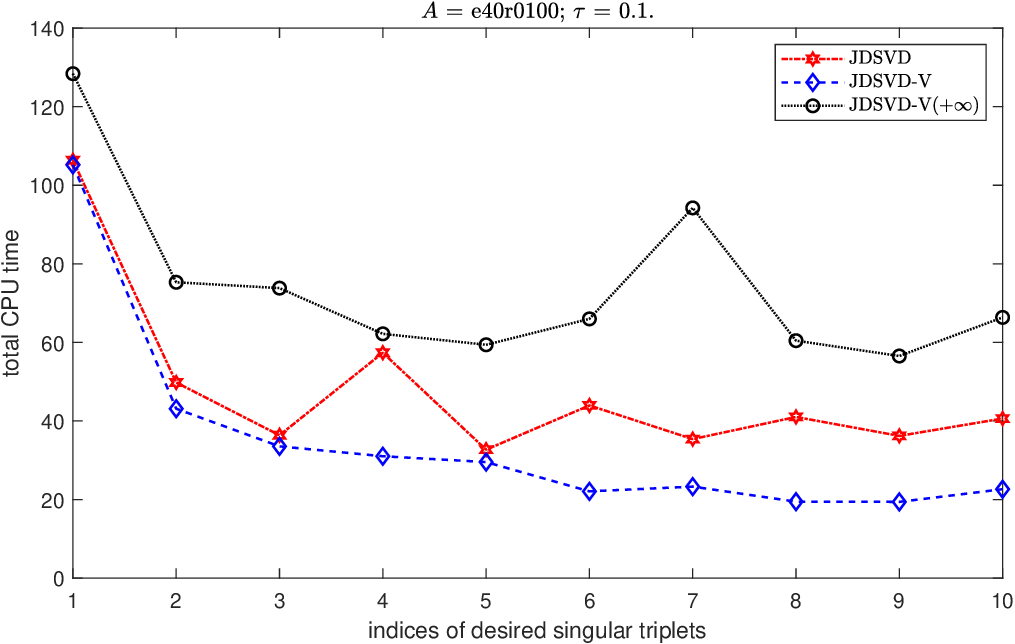}\\[0.2em]
	
	\includegraphics[width=0.48\textwidth]{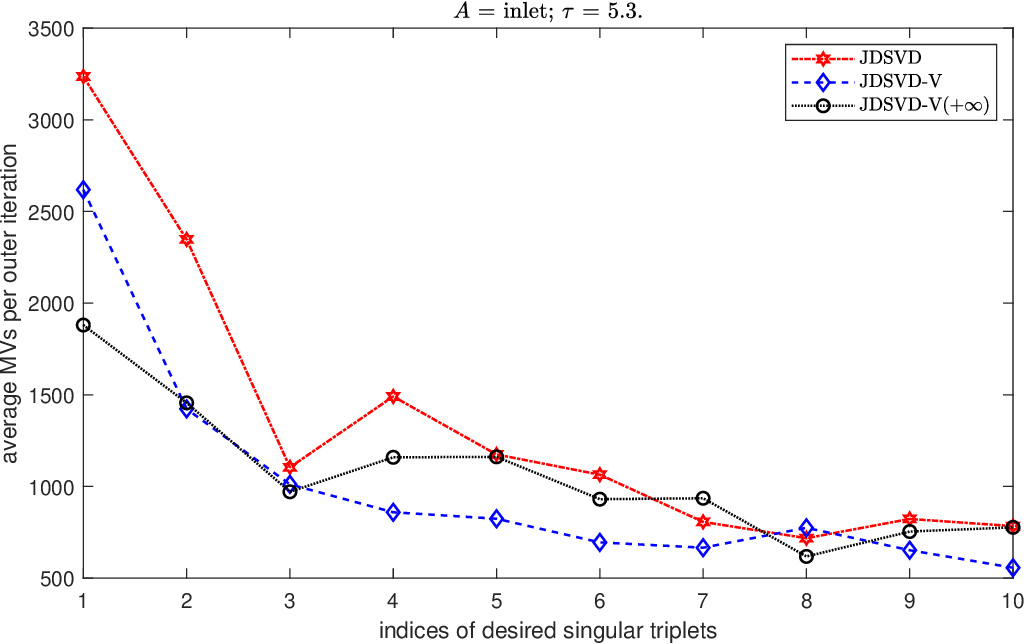}\hfill
	\includegraphics[width=0.465\textwidth]{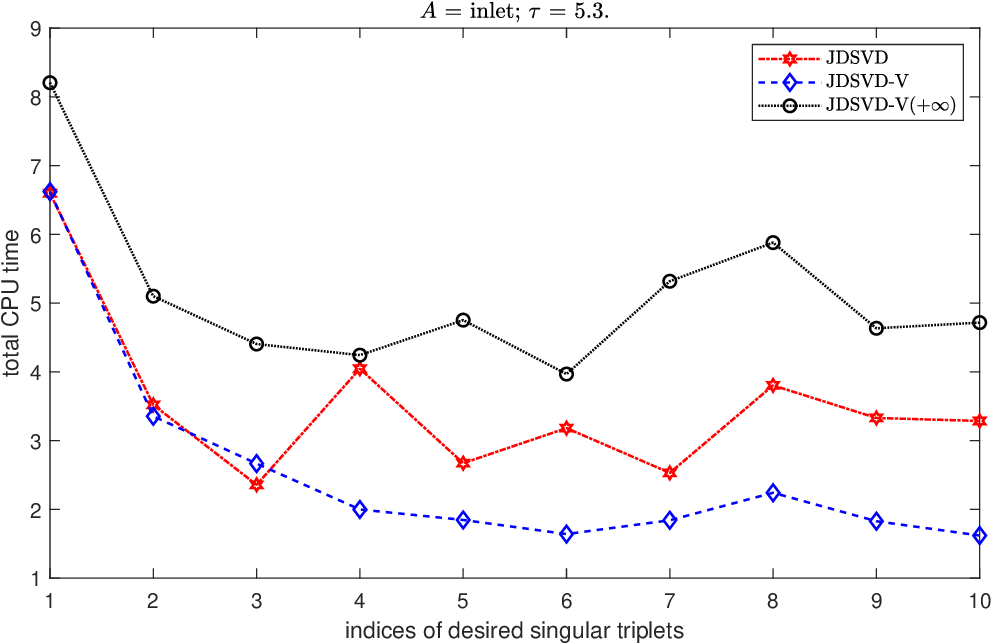}\\[0.2em]
	
	\includegraphics[width=0.48\textwidth]{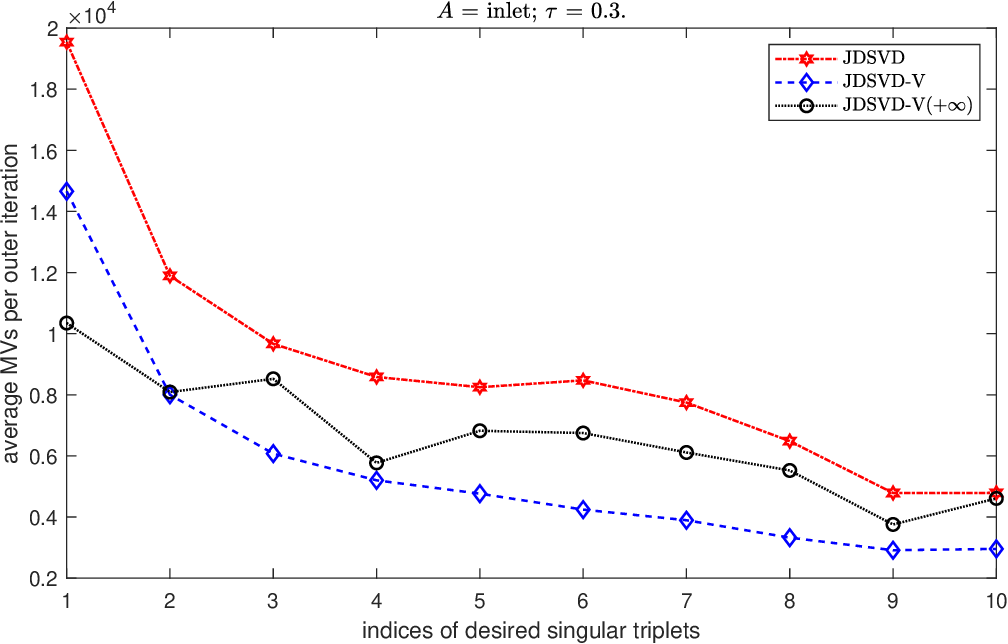}\hfill
	\includegraphics[width=0.48\textwidth]{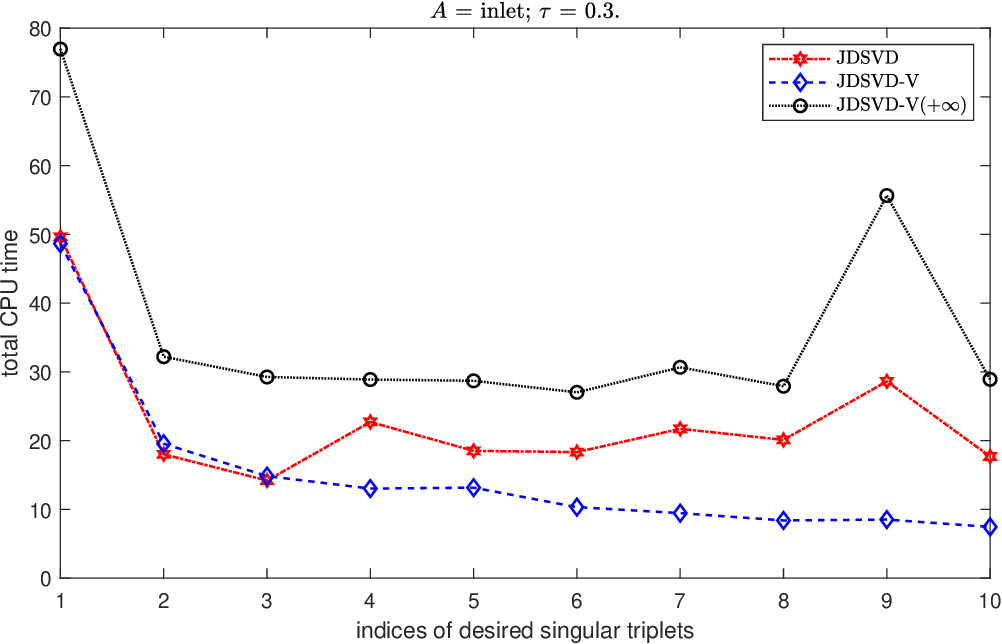}
	
	\caption{Computing the ten singular triplets of  ``e40r0100'' and ``inlet''.}\label{fig0}
\end{figure}

\begin{table}[tbhp] \small
	\caption{Results on computing the ten singular triplets of ``e40r0100'' and ``inlet'' for given targets.}\label{table0}
	\begin{center}
		\begin{tabular}{ccccccccccc}
			\toprule
			\multirow{2}{*}{$A$}&\multirow{2}{*}{$\tau$}
			&\multicolumn{3}{c}{JDSVD}
			&\multicolumn{3}{c}{JDSVD-V}
			&\multicolumn{3}{c}{JDSVD-V($+\infty$)}\\
			\cmidrule(lr){3-5} \cmidrule(lr){6-8} \cmidrule(lr){9-11}
			&&Iter &MVs &Time
			 &Iter &MVs &Time
			 &Iter &MVs &Time \\\midrule
			
			\!\!\multirow{2}{*}{e40r0100}\!\! &13
			      &145 &329687 &97.5
				  &\bf{136} &\bf{177798} &\bf{79.0}
				  &380 &580682 &190 \\
			
			 &0.1 &138 &\!\!1704453\!\! &480
			      &\bf{136} &\!\!\bf{1106373}\!\! &\bf{349}
			      &304 &\!\!2448150\!\! &743 \\[0.2em]
			
			\multirow{2}{*}{inlet}  &5.3
			      &163 &191942 &35.3
			      &\bf{126} &\bf{128975} &\bf{25.7}
			      &264 &263554 &51.2 \\
			
			 &0.3 &152 &\!\!1252510\!\! &230
			      &\bf{129} &\bf{778922} &\bf{153}
			      &307 &\!\!1904944\!\! &366 \\
			
			\bottomrule
		\end{tabular}
	\end{center}
\end{table}

For $\tau=13$ and $5.3$, the desired singular values are the
largest ones of ``e40r0100'' and ``inlet'', and those for
$\tau=0.1$ and $0.3$ are interior ones that are highly clustered
with some other singular values.
Table~\ref{table0} displays the results of the three
JDSVD type algorithms, where bold numbers mean optimal, and Figure~\ref{fig0} (left) depicts the average
MVs per outer iteration and Figure~\ref{fig0} (right) the total CPU time that they computed
each desired singular triplet.

As we can see from Figure~\ref{fig0}, for each problem,
JDSVD-V used fewer and even much fewer MVs than
JDSVD to solve the standard correction equation \eqref{correction} and \eqref{deflatcorrection}
with $\tilde m\equiv 1$, confirming the
superiority of MINRES for solving the new correction equations \eqref{precorrection}
and \eqref{deflatcorrection}. Regarding
the overall efficiency, JDSVD-V used less time to compute each desired singular triplet
of each problem.
JDSVD-V reduced $32.82\%\sim46.11\%$ of total MVs and $18.94\%\sim33.19\%$ of the
total CPU time used by JDSVD, as is seen from \cref{table0}.
Moreover, we observe from \cref{table0} that for each problem,
JDSVD-V used fewer outer iterations than JDSVD,
confirming the higher effectiveness of the new thick-restart strategy.

We observe from \cref{fig0} that for
most of the desired singular triplets of each problem, JDSVD-V($+\infty$) used
averagely fewer and even much fewer MVs and CPU time than JDSVD per outer iteration, but
it consumed more and often considerably more MVs and CPU time than JDSVD-V in total.
Moreover, JDSVD-V($+\infty$) consumed much more outer iterations than JDSVD and
JDSVD-V did, which confirms the analysis in \cref{remark10} because some of the
Ritz approximations
have poor accuracy but participate in forming the new correction equations so that
the subspaces are expanded ineffectively. As a result, the overall efficiency
of JDSVD-V($+\infty$) is considerably inferior to JDSVD in terms of the total CPU time and MVs.

Clearly, with the chosen $\tilde\varepsilon_1$ and $\tilde\varepsilon_2$,
JDSVD-V benefits much from the
new correction equations \eqref{precorrection} and \eqref{deflatcorrection}
than JDSVD for solving the standard ones.

Finally, as we see from \cref{table0}, for each algorithm and the computation of
interior singular triplets of each matrix, while the numbers
of outer iterations are comparable to
that for the computation of the largest singular triplets of the same matrix, the inner iterations
are much more because MINRES converges much more slowly for the
highly indefinite correction equations than the definite ones, which correspond to
the interior and largest singular triplets, respectively.
This confirms \cref{thm9}.


\begin{exper}
	We now compute the ten singular triplets of the last eight test matrices in
	\cref{table00} for given targets $\tau$. The
	desired singular values of ``epb2''
are not very interior, and those of the other seven matrices are interior and clustered.
\end{exper}

\begin{figure}[tbhp]
	\centering
	\includegraphics[width=0.48\textwidth]{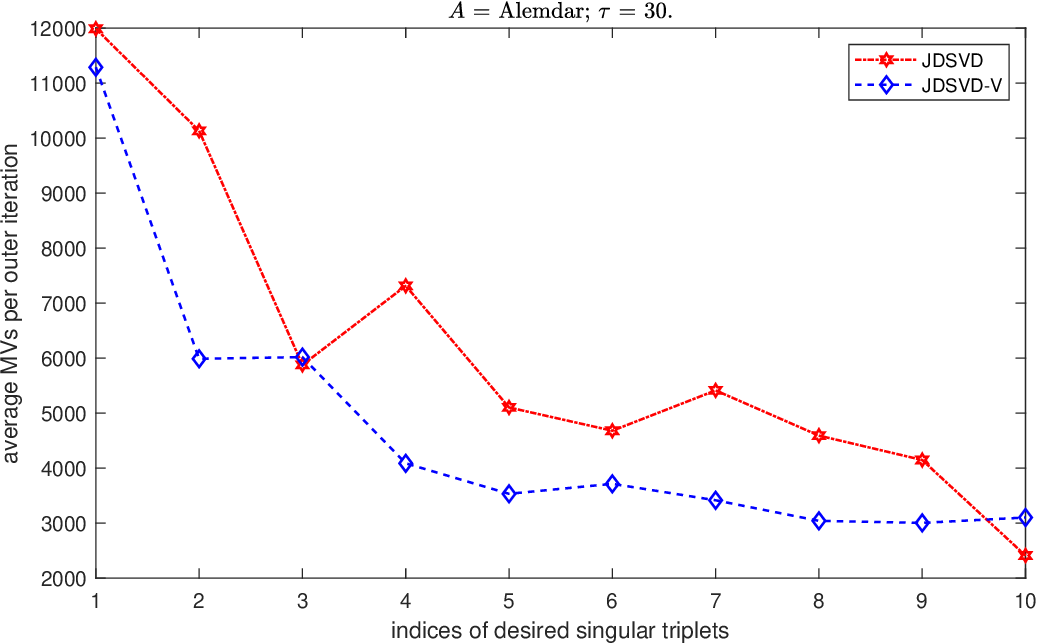}\hfill
	\includegraphics[width=0.47\textwidth]{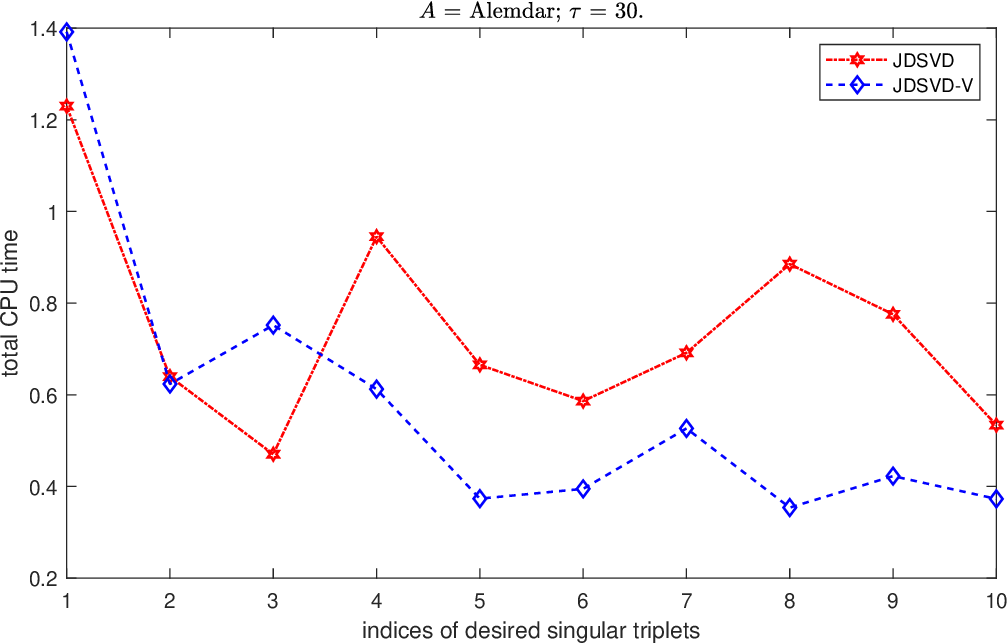}\\[0.2em]
	
	\includegraphics[width=0.48\textwidth]{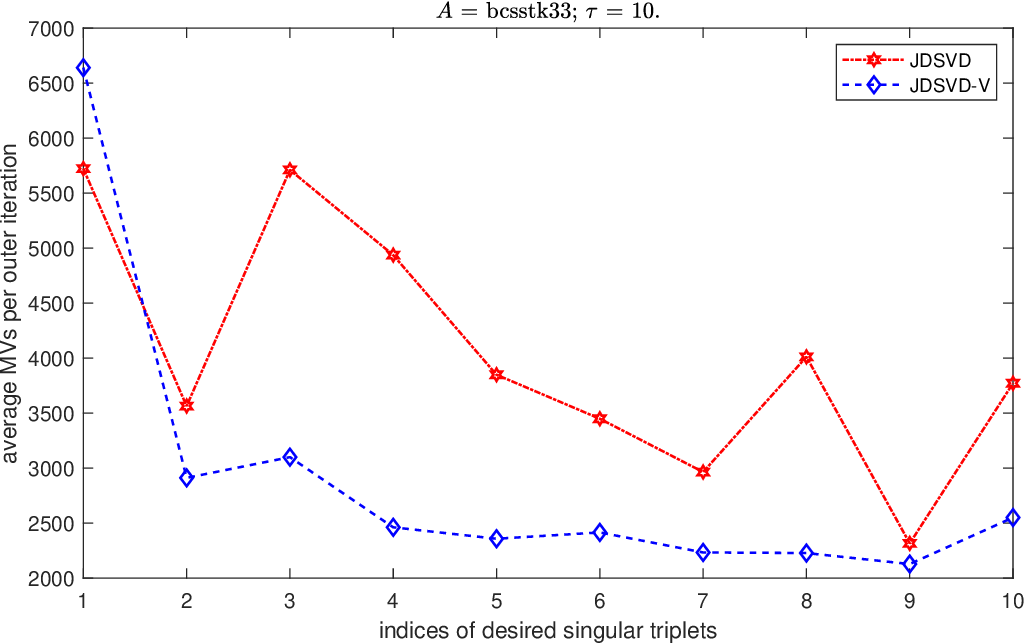}\hfill
	\includegraphics[width=0.47\textwidth]{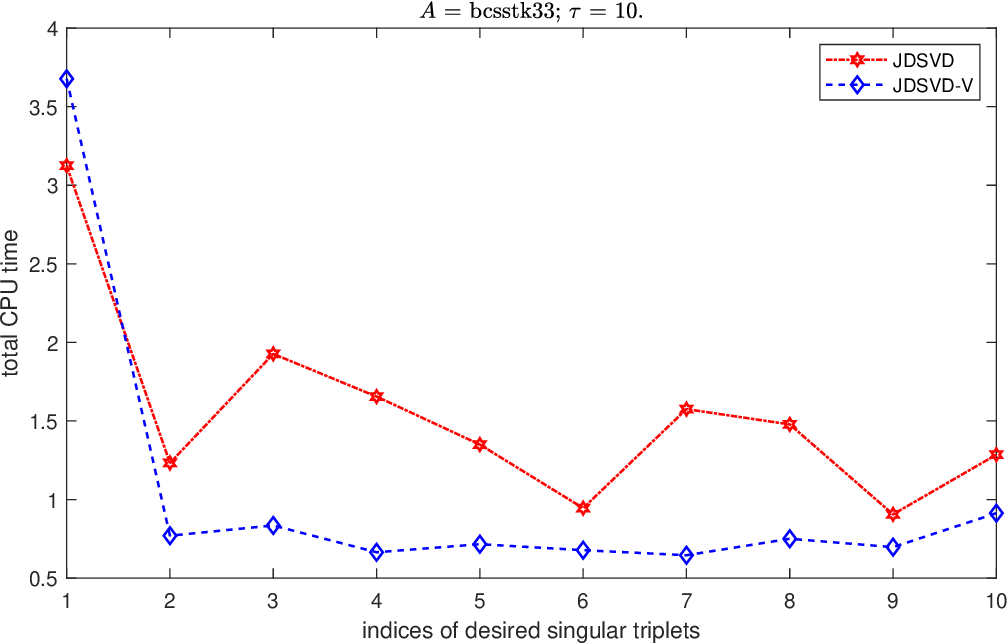}
	\caption{Computing the singular triplets of  $\mathrm{Alemdar}$ (top) and $\mathrm{bcsstk33}$ (bottom).}\label{fig1}
\end{figure}

\vspace{-0.5em}

\begin{table}[tbhp]
	\caption{Results on computing the ten singular triplets with given targets.}\label{table1}
	\begin{center}
		\begin{tabular}{cccccccc}
			\toprule
			\multirow{2}{*}{$A$}&\multirow{2}{*}{$\tau$}
			&\multicolumn{3}{c}{JDSVD}
			&\multicolumn{3}{c}{JDSVD-V}\\
			\cmidrule(lr){3-5} \cmidrule(lr){6-8}
			&&Iter &MVs &Time
			&Iter &MVs &Time \\\midrule
			
			Alemdar  &30 &165 &925964 &76.6 &\bf{135} &\bf{610402} &\bf{54.4}   \\
			cat\_ears\_3\_4 &3 &150 &1620502 &165 &\bf{143} &\bf{942971} &\bf{109}  \\
			epb2 &2.2 &123 &94836 &20.7 &\bf{62}&\bf{48842} &\bf{11.8}  \\		
			relat8 &5 &120&805104 &2.58e+3  &\bf{115} &\bf{503362} &\bf{2.10e+3} \\
			bcsstk33 &10 &159 &672383 &161 &\bf{125} &\bf{415379} &\bf{111}  \\
			dw8192 &5.4 &197 &2649603 &263 &\bf{152}&\bf{1200845} &\bf{140}  \\
			garon2 &19 &\bf{127} &1128861 &223 &135 &\bf{700079} &\bf{165}  \\
			tomographic1 &4.5 &117 &319431 &297 &\bf{105} &\bf{213754} &\bf{255}   \\
			\bottomrule
		\end{tabular}
	\end{center}
\end{table}

JDSVD-V and JDSVD succeeded in computing the desired
singular triplets of all the test matrices with ten
different sets of random starting vectors, and \cref{table1}
lists their average computational results.
Regarding the outer iterations, we see from the table that
JDSVD-V converged faster than JDSVD for almost all the test matrices,
especially for ``epb2'', for which JDSVD-V reduced nearly
half of the outer iterations needed by JDSVD,
confirming the substantial higher effectiveness of the new thick-restart strategy proposed in \cref{subsec6}.
Regarding the overall efficiency, \cref{fig1} (left)  depicts
the average MVs per outer iteration and \cref{fig1} (right) draws the total CPU time used by
JDSVD and JDSVD-V to compute each desired singular triplet of ``Alemdar'' and ``bcsstk33''.
As we can see, JDSVD-V, on average, saved $1088$ and $4209$ MVs
at each outer iteration for these two matrices, respectively,
huge improvements! It saved $23.04\%\sim60.02\%$ of the CPU time used by JDSVD for
computing the last seven and nine desired singular triplets of ``Alemdar'' and ``bcsstk33'', respectively.
For all the test matrices, JDSVD-V
reduced $33.10\%\sim54.68\%$ of MVs and
$14.24\%\sim46.70\%$ of the CPU time by JDSVD.

Obviously, for these matrices, JDSVD-V is substantially superior to JDSVD.

\begin{exper}\label{exper3}
We compare JDSVD-V\_HYBRID with PRIMME\_SVDS and the
implicitly restarted LBD algorithm (IRLBD) \cite{jia2003implicitly} when computing the ten
largest singular triplets of the test matrices ``Alemdar'' in \cref{table00}
and the seven ones from the SuiteSparse Matrix Collection \cite{davis2011university},
whose basic properties are listed in \cref{table00b}.
For PRIMME\_SVDS, we took the method in the first stage as the default
``DYNAMIC'' so that it dynamically switched between JDQMR and the GD+k
method to minimize the time, and in the second stage we
took the recommended ``JDQMR\_ETol'' method and set the
maximum block size as $1$ to make a fair comparison.
For both algorithms, we took the maximum and minimum subspace dimensions
$k_{\max}=30$ and $k_{\min}=3$ and the
same starting vectors $u_0={\sf randn}(M,1)$ and $v_0={\sf randn}(N,1)$.
The other remaining parameters in the two algorithms were set by default.
We set the target $\tau=\|A\|_e$ for JDSVD-V\_HYBRID and Sigma $=$ ``L''
for PRIMME\_SVDS and IRLBD.

\begin{table}[tbhp]
	\caption{Properties of the test matrices: part II.}\label{table00b}
	\begin{center}
		\begin{tabular}{cccccc} \toprule
			$A$&$M$&$N$&$\textit{nnz}$&$\sigma_{\max}$&$\sigma_{\min}$  \\ \midrule
			3elt &4720 &4720 &27444 &6.03 &2.12e-3 \\
			plddb &3069 &5049 &10839 &1.43e+2 &1.16e-2 \\
			bas1lp &5411 &9825 &587775 &1.49e+3 &3.62e-2 \\	
			G66 &9000 &9000 &36000 &3.58 &2.14e-4 \\
			stat96v4 &3173 &63076 &491336 &1.96e+1 &2.07e-3 \\	
			atmosmodl &1489752&1489752&10319760&6.21e+5 &5.42e+2 \\	
			cage15 &5154859&5154859&99199551&1.02 &8.56e-2 \\
			\bottomrule
		\end{tabular}
	\end{center}
\end{table}

Notice that MVs dominates the overall
efficiency of each algorithm. Our JDSVD-V\_HYBRID and IRLBD were written in
{\sc Matlab} language and PRIMME\_SVDS was optimally programmed in C language,
the latter thus should be faster than JDSVD-V\_HYBRID and IRLBD
in terms of CPU time when comparable MVs are used.
\end{exper}

\begin{table}[tbhp]
	\caption{Results on computing the ten largest singular triplets.}\label{table3}
	\begin{center}
		\begin{tabular}{ccccccc}
			\toprule
			\multirow{2}{*}{$A$}&\multicolumn{2}{c}{IRLBD}&\multicolumn{2}{c}{PRIMME\_SVDS} &\multicolumn{2}{c}{JDSVD-V\_HYBRID} \\
			\cmidrule(lr){2-3} \cmidrule(lr){4-5} \cmidrule(lr){6-7}
			 &MVs  &Time &MVs  &Time  &MVs  &Time   \\\midrule
			Alemdar   &\bf{426}&\bf{1.18e-1} &1758&2.08e-1 &1694&2.58e-1   \\
			3elt      &\bf{753}&\bf{1.67e-1}  &2210&1.92e-1&2058&2.11e-1   \\
			plddb     &122820&24.2    &\bf{3608}&\bf{1.86e-1} &4102&3.30e-1 \\ 	
			bas1lp    &\bf{88}&\bf{5.76e-2} &299&1.10e-1 &413&2.14e-1    \\
			G66       &\bf{729}&\bf{1.82e-1} &2297&2.38e-1  &2073&3.13e-1  \\
			stat96v5  &\bf{255}&\bf{1.95e-1} &1127&2.52e-1 &1078&5.07e-1    \\
			atmosmodl  &84778&3.48e+3 &13013&\bf{334} &\bf{10583}&552\\
			cage15    &\bf{3779} &\bf{613} &5555 &796 &4884 &1.03e+3    \\
			\bottomrule
		\end{tabular}	
	\end{center}
\end{table}

\cref{table3} displays the average computational results
over ten runs for the three algorithms, where bold numbers mean optimal.
As we can see, JDSVD-V\_HYBRID
consumed slightly to moderately fewer MVs than PRIMME\_SVDS
for most of the eight test matrices, except for ``plddb'' and ``bas1lp'',
for which the latter used considerably fewer MVs than
the former. IRLBD outperformed JDSVD-V\_HYBRID and PRIMME\_SVDS
for six of the eight matrices by
using significantly fewer MVs and less CPU time.
We notice that some of the desired singular values of
``plddb'' and  ``atmosmodl'' are very clustered, which caused the slow convergence
of IRLBD, and
JDSVD-V\_HYBRID and PRIMME\_SVDS were substantially superior to IRLBD.
JDSVD-V\_HYBRID and PRIMME\_SVDS were competitive in terms of
MVs and CPU time for most of the test matrices; PRIMME\_SVDS outperformed JDSVD-V\_HYBRID for  ``bas1lp'',
and JDSVD-V\_HYBRID outmatched PRIMME\_SVDS for ``atmosmodl'' in terms of MVs.

In summary, for computing several largest singular triplets of
a large scale matrix,
JDSVD-V\_HYBRID is at least competitive with PRIMME\_SVDS; IRLBD is
generally preferable if the largest singular values are {\em not} clustered,
otherwise the other two algorithms may be advantageous.

\begin{exper}
Setting $\tau=0$ for JDSVD-V\_HYBRID and Sigma $=$ ``S'' for PRIMME\_SVDS,
we compute the ten smallest singular triplets of the eight matrices in
Experiment~\ref{exper3} using JDSVD-V\_HYBRID, PRIMME\_SVDS and PRIMME\_SVDS
with the block Jacobi preconditioners (PRIMME\_SVDS\_J) with block size $10$ and $100$
for the first six and last two matrices, respectively.
\end{exper}
	
\begin{table}[tbhp]	
	\caption{Results on computing the ten smallest singular triplets.}\label{table4}
	\begin{center}
		\begin{tabular}{ccccccc}
			\toprule
			\multirow{2}{*}{$A$}&\multicolumn{2}
			{c}{\!\!PRIMME\_SVDS\!\!}&\multicolumn{2}{c}{\!\!PRIMME\_SVDS\_J\!\!}&\multicolumn{2}{c}{JDSVD-V\_HYBRID} \\
			\cmidrule(lr){2-3} \cmidrule(lr){4-5}  \cmidrule(lr){6-7}
			  &MVs  &Time &MVs  &Time  &MVs  &Time   \\\midrule
			\!\!Alemdar\!\!    &368860&25.5  &229956&65.6  &\bf{209840}&\bf{12.5} \\
			
			3elt       &153168&7.87  &2362024&1.52e+3&\bf{130341}&\bf{6.40}   \\
			plddb       &33575&1.42  &\bf{22772}&6.64 &26852&\bf{0.66}  \\
			bas1lp      &443433&96.6  &\!\!\bf{283159}\!\!&150 &344999&\bf{62.7} \\
			G66          &\!\!8325715\!\!&\!\!861\!\!  &\!\!24523528\!\!&7.27e+3 &\bf{266884}&\bf{27.4}\\
			\!\!stat96v5\!\!     &23121&5.88  &587180&1.73e+3 &\bf{20046}&\bf{4.25}\\
			\!\!atmosmodl\!\!  &358703&\bf{5.30e+3} &\!\!\bf{320026}\!\!&\!\!1.94e+4\!\!&504452&\!\!5.65e+3\!\!  \\
			cage15    &5392&\bf{737}  &\bf{2867}&\!\!3.44e+3\!\!  &4969&776  \\
			\bottomrule
		\end{tabular}
	\end{center}
\end{table}

\cref{table4} reports the average performance of the three algorithms over
ten runs.
As we see from the table, except ``atmosmodl'' for which JDSVD-V\_HYBRID
used $40.63\%$ more MVs than PRIMME\_SVDS, for the other seven matrices,
JDSVD-V\_HYBRID performed
considerably better than PRIMME\_SVDS in terms of the total MVs and/or CPU time;
on average, JDSVD-V\_HYBRID saved $15.34\%\sim75.78\%$ of the MVs used by
PRIMME\_SVDS for ``Alemdar'', ``3elt'',
``plddb'', ``bas1lp'' and ``stat96v5'', and it even used
much less CPU time than PRIMME\_SVDS.
Particularly, for ``G66'', JDSVD-V\_HYBRID converged fast and
computed all the desired singular triplets correctly, while
PRIMME\_SVDS failed to converge 
three times out of the ten runs,
and only output four to eight converged singular triplets.

When block Jacobi preconditioners were used in PRIMME\_SVDS for the
correction equations, they only accelerated the convergence
of the inner iterations for ``Alemdar'', ``plddb'', ``bas1lp'',
and ``cage15'' but
did not help for the other matrices in terms of the
MVs and the CPU time. Overall, JDSVD-V\_HYBRID
outmatched PRIMME\_SVDS\_J considerably in terms of MVs and
CPU time, except for the last two matrices where JDSVD-V\_HYBRID used more MVs but much shorter time.
These indicate that the application of the block Jacobi preconditioners with {\em bigger} size 100 was
time consuming though they were effective for the matrices ``atmosmodl'' and ``cage15''.
Therefore, JDSVD-V\_HYBRID outperforms PRIMME\_SVDS and
PRIMME\_SVDS\_J greatly when computing the smallest
singular triplets.

\section{Conclusions}\label{sec:7}

We have derived new correction equations in JDSVD,
and proved that their solutions expand the subspaces as
effectively as those of the standard correction equations. The resulting
JDSVD-V method is a new variant of the JDSVD method with a new
correction equation solved at each outer iteration.
We have established convergence results on MINRES for the new and standard correction
equations, and shown that MINRES can converge much faster for solving
the new correction equations than it does for the standard ones.

We have presented a practical strategy that dynamically selects
approximate singular triplets from current available ones and
forms a new correction equation at each outer iteration.
Finally, we have developed a new thick-restart JDSVD-V algorithm with
deflation and purgation to compute several singular triplets.
Numerical experiments have illustrated that the new thick-restart JDSVD-V algorithm
outperforms the standard thick-restart JDSVD algorithm considerably
in terms of the outer iterations, MVs and CPU time.
JDSVD-V\_HYBRID is at least competitive with PRIMME\_SVD
in \cite{wu16primme} for computing the largest singular triplets
and outperforms the latter substantially for computing the smallest ones in terms of MVs and reliability.

\section*{Declarations}

The two authors declare that they have no conflict of interest, and they read and approved the final manuscript.


\begin{thebibliography}{10}

\bibitem{berry1992large}
{\sc M.~W. Berry}, {\em Large-scale sparse singular value decompositions}, Int.
  J. Supercomput. Appl., 6 (1992), pp.~13--49,
  \url{https://doi.org/10.1177/109434209200600103}.

\bibitem{davis2011university}
{\sc T.~A. Davis and Y.~Hu}, {\em The {U}niversity of {F}lorida sparse matrix
  collection}, ACM Trans. Math. Software, 38 (2011), pp.~1--25,
  \url{https://doi.org/10.1145/2049662.2049663}.
\newblock Data available online at
  \url{http://www.cise.ufl.edu/research/sparse/matrices/}.

\bibitem{deSturler2002ImprovingTC}
{\sc E.~de~Sturler}, {\em Improving the convergence of the {J}acobi-{D}avidson
  algorithm}, Preprint, Department of Computer Science, University of Illinois
  at Urbana-Champaign,  (2002),
  \url{https://api.semanticscholar.org/CorpusID:17619489}.

\bibitem{genseberger1999alternative}
{\sc M.~Genseberger and G.~L. Sleijpen}, {\em Alternative correction equations
  in the {J}acobi-{D}avidson method}, Numer. Linear Algebra Appl., 6 (1999),
  pp.~235--253.

\bibitem{goldenberg2019}
{\sc S.~Goldenberg, A.~Stathopoulos, and E.~Romero}, {\em A {G}olub--kahan
  {D}avdson method for accurately computing a few singular triplets of large
  sparse matrices}, SIAM J. Sci. Comput., 41 (2019), pp.~A2172--A2192.

\bibitem{golub2012matrix}
{\sc G.~H. Golub and C.~F. van Loan}, {\em Matrix Computations, 4th Ed.}, The
  John Hopkins University Press, Baltimore, 2012.

\bibitem{greenbaum1997iterative}
{\sc A.~Greenbaum}, {\em Iterative Methods for Solving Linear Systems}, SIAM,
  Philadelphia, PA, 1997.

\bibitem{hernandes2008robust}
{\sc V.~Hern{\'a}ndes, J.~E. Rom{\'a}n, and A.~Tom{\'a}s}, {\em A robust and
  efficient parallel {SVD} solver based on restarted {L}anczos
  bidiagonalization}, Electron. Trans. Numer. Anal., 31 (2008), pp.~68--85.
\newblock Resource available online at \url{https://etna.math.kent.edu}.

\bibitem{hochstenbach2001jacobi}
{\sc M.~E. Hochstenbach}, {\em A {J}acobi--{D}avidson type {SVD} method}, SIAM
  J. Sci. Comput., 23 (2001), pp.~606--628,
  \url{https://doi.org/10.1137/S1064827500372973}.

\bibitem{hochstenbach2004harmonic}
{\sc M.~E. Hochstenbach}, {\em Harmonic and refined extraction methods for the
  singular value problem, with applications in least squares problems}, BIT, 44
  (2004), pp.~721--754, \url{https://doi.org/10.1007/s10543-004-5244-2}.

\bibitem{hochstenbach2008harmonic}
{\sc M.~E. Hochstenbach and G.~L. Sleijpen}, {\em Harmonic and refined
  {R}ayleigh--{R}itz for the polynomial eigenvalue problem}, Numer. Linear
  Algebra Appl., 15 (2008), pp.~35--54, \url{https://doi.org/10.1002/nla.562}.

\bibitem{huang2019inner}
{\sc J.~Huang and Z.~Jia}, {\em On inner iterations of {J}acobi--{D}avidson
  type methods for large {SVD} computations}, SIAM J. Sci. Comput., 41 (2019),
  pp.~A1574--A1603, \url{https://doi.org/10.1137/18M1192019}.

\bibitem{huang2022harmonic}
{\sc J.~Huang and Z.~Jia}, {\em Two harmonic {J}acobi--{D}avidson methods for
  computing a partial generalized singular value decomposition of a large
  matrix pair}, J. Sci. Comput., 93 (2022), p.~41,
  \url{https://doi.org/10.1007/s10915-022-01993-7}.

\bibitem{huang2023cross}
{\sc J.~Huang and Z.~Jia}, {\em A cross-product free {J}acobi--{D}avidson type
  method for computing a partial generalized singular value decomposition of a
  large matrix pair}, J. Sci. Comput., 94 (2023), p.~3,
  \url{https://doi.org/10.1007/s10915-022-02053-w}.

\bibitem{jia2005convergence}
{\sc Z.~Jia}, {\em The convergence of harmonic {R}itz values, harmonic {R}itz
  vectors and refined harmonic {R}itz vectors}, Math. Comput., 74 (2005),
  pp.~1441--1456, \url{https://doi.org/10.1090/S0025-5718-04-01684-9}.

\bibitem{jia2014inner}
{\sc Z.~Jia and C.~Li}, {\em Inner iterations in the shift-invert residual
  {A}rnoldi method and the {J}acobi--{D}avidson method}, Sci. China Math., 57
  (2014), pp.~1733--1752, \url{https://doi.org/10.1007/s11425-014-4791-5}.

\bibitem{jia2003implicitly}
{\sc Z.~Jia and D.~Niu}, {\em An implicitly restarted refined bidiagonalization
  {L}anczos method for computing a partial singular value decomposition}, SIAM
  J. Matrix Anal. Appl., 25 (2003), pp.~246--265,
  \url{https://doi.org/10.1137/S0895479802404192}.

\bibitem{jia2010refined}
{\sc Z.~Jia and D.~Niu}, {\em A refined harmonic {L}anczos bidiagonalization
  method and an implicitly restarted algorithm for computing the smallest
  singular triplets of large matrices}, SIAM J. Sci. Comput., 32 (2010),
  pp.~714--744, \url{https://doi.org/10.1137/080733383}.

\bibitem{jiazhang2023a}
{\sc Z.~Jia and K.~Zhang}, {\em A {FEAST} {SVD}solver based on
  {C}hebyshev--{J}ackson series for computing partial singular triplets of
  large matrices}, J. Sci. Comput., 97 (2023), p.~21,
  \url{https://doi.org/10.1007/s10915-023-02342-y}.

\bibitem{jiazhang2023b}
{\sc Z.~Jia and K.~Zhang}, {\em An augmented matrix-based {CJ-FEAST}
  {SVD}solver for computing a partial singular value decomposition with the
  singular values in a given interval}, SIAM J. Matrix Anal. Appl., 45 (2024),
  pp.~24--58, \url{https://doi.org/10.1137/23M1547500}.

\bibitem{kokiopoulou2004computing}
{\sc E.~Kokiopoulou, C.~Bekas, and E.~Gallopoulos}, {\em Computing smallest
  singular triplets with implicitly restarted {L}anczos bidiagonalization},
  Appl. Numer. Math., 49 (2004), pp.~39--61,
  \url{https://doi.org/10.1016/j.apnum.2003.11.011}.

\bibitem{larsen2001combining}
{\sc R.~M. Larsen}, {\em Combining implicit restarts and partial
  reorthogonalization in {L}anczos bidiagonalization}.
\newblock Presentation at U.C. Berkeley, sponsored by Scientific Computing and
  Computational Mathematics, Stanford University, 2001.

\bibitem{niu2012harmonic}
{\sc D.~Niu and X.~Yuan}, {\em A harmonic {L}anczos bidiagonalization method
  for computing interior singular triplets of large matrices}, Appl. Math.
  Comput., 218 (2012), pp.~7459--7467,
  \url{https://doi.org/10.1016/j.amc.2012.01.013}.

\bibitem{parlett1998symmetric}
{\sc B.~N. Parlett}, {\em The Symmetric Eigenvalue Problem}, SIAM,
  Philadelphia, PA, 1998.

\bibitem{saad2003}
{\sc Y.~Saad}, {\em Iterative Methods for Sparse Linear Systems, 2nd ed.},
  SIAM, Philadelphia, PA, 2003.

\bibitem{sorensen1992implicit}
{\sc D.~C. Sorensen}, {\em Implicit application of polynomial filters in a
  $k$-step {A}rnoldi method}, SIAM J. Matrix Anal. Appl., 13 (1992),
  pp.~357--385, \url{https://doi.org/10.1137/0613025}.

\bibitem{stathsisc2007}
{\sc A.~Stathopoulos}, {\em Nearly optimal preconditioned methods for
  {H}ermitian eigenproblems under limited memory. {P}art {I}: Seeking one
  eigenvalue}, SIAM J. Sci. Comput., 29 (2007), pp.~481--514,
  \url{https://doi.org/10.1137/050631574}.

\bibitem{stath1998}
{\sc A.~Stathopoulos, Y.~Saad, and K.~Wu}, {\em Dynamic thick restarting of the
  {D}avidson, and the implicitly restarted {A}rnoldi methods}, SIAM J. Sci.
  Comput., 19 (1998), pp.~227--245,
  \url{https://doi.org/10.1137/S1064827596304162}.

\bibitem{stewart2001matrix}
{\sc G.~W. Stewart}, {\em Matrix Algorithms II: Eigensystems}, SIAM,
  Philadelphia, PA, 2001.

\bibitem{vandervorst}
{\sc H.~Van~der Vorst}, {\em Computational Methods for Large Eigenvalue
  Problems}, Elsvier, Holland, 2002.

\bibitem{wu16primme}
{\sc L.~Wu, R.~Romero, and A.~Stathopoulos}, {\em {PRIMME}\_{SVDS}: {A}
  high--performance preconditioned {SVD} solver for accurate large--scale
  computations}, SIAM J. Sci. Comput., 39 (2017), pp.~S248--S271,
  \url{https://doi.org/10.1137/16M1082214}.

\bibitem{wu2015preconditioned}
{\sc L.~Wu and A.~Stathopoulos}, {\em A preconditioned hybrid {SVD} method for
  accurately computing singular triplets of large matrices}, SIAM J. Sci.
  Comput., 37 (2015), pp.~S365--S388, \url{https://doi.org/10.1137/140979381}.

\end{thebibliography}

\end{document}